\documentclass[10pt]{amsart}

\usepackage{geometry,graphicx,amssymb,amsmath,amsbsy,eucal,amsfonts,mathrsfs,amscd,bm,tcolorbox}

\numberwithin{equation}{section}

\allowdisplaybreaks[3]

\newtheorem{theorem}{Theorem}[section]
\newtheorem{lemma}[theorem]{Lemma}

\newtheorem{corollary}[theorem]{Corollary}
\newtheorem{proposition}[theorem]{Proposition}
\theoremstyle{definition}

\theoremstyle{remark}

\newcommand{\jump}[1]{\left[\hspace{-0.025in}\left[#1\right]\hspace{-0.025in}\right]}

\newcommand{\mVh}{\mathring{V}_h}
\newcommand{\Pe}{Pe}
\newcommand{\Co}{Co}

\title[]{Implicit-explicit multistep formulations for finite element
  discretisations using continuous interior penalty }

\author[]{Erik Burman}
\address{Department of Mathematics
University College London
}
\email{e.burman@ucl.ac.uk}

\author[]{Johnny Guzm\'an}
\address{Division of Applied Mathematics
Brown University
Providence, USA}
\email{johnny\_guzman@brown.edu}

\begin{document}

\maketitle

\begin{abstract}
We consider a finite element method with symmetric stabilisation for
the discretisation of the transient convection--diffusion
equation. For the time-discretisation we consider either the second order
backwards differentiation formula or the Crank-Nicolson method. Both the convection term and the
associated stabilisation are treated explicitly using an extrapolated
approximate solution. We prove stability of the method and the
$\tau^2 + h^{p+{\frac12}}$ error estimates for the $L^2$-norm under
either the standard hyperbolic CFL condition, when piecewise affine ($p=1$)
approximation is used, or in the case of finite element approximation of order $p \ge 1$, a stronger, so-called $4/3$-CFL, i.e. $\tau \leq C
h^{4/3}$. The theory is illustrated with some numerical examples.
\end{abstract}



\section{Introduction}
In the computational solution of convection--diffusion problems it is
highly advantageous to treat the convection term explicitly and the
diffusion term implicitly. Although there is an important literature
on the topic there seems to be very few works that show that
implicit-explicit (IMEX) 
methods are robust under finite element discretisation both in the convection and the diffusion dominated
regimes. Indeed most works on IMEX methods with finite element
discretisations assume that the diffusion
dominates, typically leading to an assumption that the product of the
Courant number and the Peclet number is small \cite{Crouz80,Va80,ACM98,Akri13, Akri18}. Most other works on
IMEX methods typically consider a stability region type analysis that
is unsuitable for a quantitative finite element analysis \cite{ARW95,
  ARS97, Hund98, Hund01}. To pass from the analysis of the
semi-discrete case to a fully discrete case typically requires the use
of energy methods as advocated in \cite{LT98}. This was the route taken in 
\cite{BE12} where a second order implicit explicit Runge-Kutta method was
considered for $H^1$-conforming finite element approximations of convection--diffusion
equations using interior penalty to guarantee stability in the high
Peclet regime and more recently \cite{WZS19} where a local discontinuous
Galerkin methods with 2nd and 3rd order Runge-Kutta IMEX schemes are considered. Explicit Runge-Kutta methods have been very successful for the
approximation of hyperbolic equations often in combination with
discontinuous Galerkin methods and for this case there is a very large
literature \cite{CS89, ZS04,BEF10,ZS10}, to name a few. In particular the mass matrix
is block diagonal allowing for high order explicit time stepping, this
however is no longer the case for IMEX methods where the system matrix
associated to diffusion couples globally. Moreover if the target application is
the incompressible Navier-Stokes' equations, the explicit methods are
unsuitable and methods with many implicit
stages may become costly. It is not clear that the analysis of
\cite{BE12} generalises to this case, since it is assumed there that
the operator treated implicitly is symmetric and
elliptic. For a second order scheme the RK IMEX method has three
stages and is known to impose compatibility
conditions on the exact solution for consistency to
hold (see \cite[Section 3]{BE12}, for a discussion in the case of
convection--diffusion equations).  Finite element-IMEX methods with extrapolation were  considered as early as 1970 by Douglas and Dupont \cite{douglas1970galerkin}, however,  they considered  diffusion dominated problems.
Implicit-explicit methods with extrapolation for the
incompressible Navier-Stokes equations and backward differentiation
used for time discretisation were analysed in \cite{BDV82}, and some
other multi-step IMEX methods together with the local Galerkin method
was considered in \cite{WLZS19},
but the mesh Peclet number, defined by 
$$Pe := \frac{|u| h}{\mu}$$
where $u$ denotes the fluid velocity, $\mu$ the viscosity coefficient
and $h$ the local mesh size,
is assumed to be small. In many applications such as
  large eddy simulation or under resolved DNS it may not be possible
  to satisfy such a condition, nevertheless for such problems, in which
  convective effects are strong, the use of IMEX schemes is very
  attractive, since the nonlinearity and stabilization terms are
  handled explicitly, whereas the velocity-pressure coupling, which is
  implicit can be solved using optimized methods for saddle point
  problems.
As a first step towards IMEX schemes for the equations of
incompressible flow we will in this paper consider the
convection-diffusion equation and analyse some known IMEX schemes with
respect to their stability properties for varying mesh Peclet number.
For an IMEX scheme to be stable for high mesh
Peclet number it has to degenerate to a stable explicit scheme in the
limit of vanishing diffusion. Such time integrator are typically
characterized by nontrivial imaginary stability boundary. Examples are
given by Adams-Bashforth (AB) integrators of higher order such as AB3,
AB4, AB7 and AB8 (see \cite{GFR15}). Unfortunately IMEX schemes
designed using methods popular for the solution of incompressible flow
problems such as the Crank-Nicolson scheme or the second order
backward differentiation scheme do not enjoy this
property, if the convection is treated explicitly using
extrapolation. Indeed in this case the limit schemes are the second order Adams-Bashforth scheme for pure
transport and the extrapolated Gear scheme \cite{Va80} (BDF2 with
extrapolated convcetion). Both have
trivial imaginary stability boundary and would therefore seem
unsuitable candidates for Peclet robust IMEX methods. Nevertheless in
this contribution we will consider these two schemes with a finite
element space discretization stabilized using continuous interior
penalty. Observe that the explicit treatment of the stabilization
is appealing since it avoids having to
handle the extended stencil on the level of the linear solver \cite{BF09}.
We use energy methods to prove that they are stable, irrespective of
the Peclet number,
under suitable CFL conditions. The stability however depends on the space
discretization. Both the polynomial order of the approximation space
and the stabilization of the convection operator come into play. Using
the additional stability of the gradient penalty operator we show
optimal error estimates for the material derivative and $L^2$-error
error estimates with the classical order $O(h^{p+\frac12})$ (where $p$
is the polynomial degree), which is known
to be the best that can be obtained for continuous FEM in the general case.
 Observe also that since our results are robust with respect to the Peclet
number they remain valid for the case of vanishing diffusion, i.e. the pure transport
equation. So the present work also give the first analysis of these
explicit methods together with stabilized FEM for the transport equation.

We only consider the case of continuous approximation spaces herein,
but the analysis carries over to the case of discontinuous Galerkin
symmetric interior penalty methods, with upwind fluxes in a straightforward fashion (we
refer to \cite{BEF10} for a uniform treatment of continuous and
discontinuous Galerkin methods in the case of explicit Runge-Kutta
methods).

The outline of the paper is as follows. In the next section we
introduce our model problem, define the finite element spaces and
prove some technical results. In section \ref{sec:bdf2} we introduce
the BDF2-IMEX method  and derive stability results in all regimes
and for all polynomial orders. This allows us to prove a priori error
estimates in section \ref{sec:apriori}. Here we consider the high
Peclet regime only and derive error estimates for the $L^2$-error at
final time and the error in the material derivative over the
space-time domain. In section \ref{sec:crank} we introduce the Crank-Nicolson IMEX
scheme and prove stability estimates in all regime and for all
polynomial orders. Note that it is then straightforward to derive similar error
estimates as for the BDF2-IMEX scheme in the high Peclet regime for
the Crank-Nicolson scheme and
also optimal estimates in $l^2(0,T;H^1(\Omega))$, or
$l^\infty(0,T;L^2(\Omega))$ for both methods, but to keep the length of the paper reasonable these
results have not been included here. The paper finishes with some
numerical experiments (section \ref{sec:numerics}), validating the theory and
showing the robustness of the methods in the presence of non-smooth data.
\section{Preliminaries}

\subsection{Convection--diffusion problem}
Let $\Omega \subset \mathbb{R}^d$, $d=1,2,3$ be an open polygonal
domain with boundary $\partial \Omega$ and outward pointing normal $n$. Let $I:= (0,T)$ and denote the space time domain by $Q =
\Omega \times I$.
We consider the convection diffusion equation,
\begin{subequations}
\begin{alignat}{2}\label{eq:convdiff}
\partial_t u + \beta \cdot \nabla u - \mu \Delta u = & f \quad && \mbox{ in }
                                                        Q \\
u(\cdot, 0) =& u_0 \quad && \mbox{ in }
                                                        \Omega \\
u =& 0 \quad && \mbox{ on } \partial \Omega.
\end{alignat}
\end{subequations}
Here $f \in L^2(\Omega)$, $u_0 \in H^1_0(\Omega)$, $\beta \in
[H^1_0(\Omega)\cap W^{1,\infty}(\Omega)]^d$, with $\nabla \cdot
\beta=0$, $\beta \cdot n\vert_{\partial \Omega} = 0$. This is a parabolic
problem and it is known to admit a unique solution in
$L^2(0,T;H^1_0(\Omega)) \cap L^\infty(0,T;L^2(\Omega))$.
We define the forms
\[
(u,v)_\Omega :=\int_\Omega uv ~ \mbox{d}x, \quad c(w,v) :=\int_\Omega \beta \cdot \nabla w v ~\mbox{d}x,\quad a(w,v) := \int_\Omega \mu \nabla w \cdot 
\nabla v~\mbox{d}x.
\]
Assuming sufficient smoothness of the solution the equation may then
be cast on the weak formulation,
\begin{subequations}\label{weaku}
\begin{alignat}{2}
(\partial_t u, v)_\Omega + c(u,v)  + a(u,v) =& (f,v)_{\Omega}, \quad && \forall v \in
H^1_0(\Omega), \, t>0, \\
(u(\cdot,0),v)_\Omega =& (u_0,v)_\Omega, \quad  && \forall v \in
H^1_0(\Omega).
\end{alignat}
\end{subequations}
We will use the following two norms $\| \cdot\|^2=(\cdot, \cdot)_{\Omega}$ and $\|v\|_{\infty}=\sup_{ x \in \overline{\Omega}} |v(x)|$. 

\subsection{Finite element spaces and bilinear forms}
Let $\{\mathcal{T}_h\}_h$ denote a family of shape regular, quasi
uniform, triangulation of  $\Omega$ into simplices. The set of interior faces
of $\mathcal{T}_h$ will be denoted $\mathcal{F}$. Let $V_h$ denote the
space of continuous finite element functions of polynomial degree less
than or equal to $p$:
\begin{equation*}
V_h:=\{ v \in H^1(\Omega): v\in \mathcal{P}_p(T), \forall T \in \mathcal{T}_h\}. 
\end{equation*}
 We also consider the space with homogenuous boundary conditions $\mVh = H^1_0(\Omega) \cap V_h$.  We let $\pi_h$ be the $L^2$ projection onto $\mVh$ given by
 \begin{equation}\label{eq:L2proj}
 (\pi_h w, v_h)_{\Omega}=(w, v_h)_{\Omega} \quad \forall v_h \in \mVh.  
 \end{equation}
 We recall the following approximation estimate that holds on
 quasi-uniform meshes
\begin{equation}\label{eq:L2approx}
\|u - \pi_h u\| + h \|\nabla (u - \pi_h u) \| \leq C h^{p+1}|u|_{H^{p+1}(\Omega)}.
\end{equation}
We will also make use of the piece-wise constant space 
\begin{equation*}
W_h:=\{ v \in L^2(\Omega): v\in \mathcal{P}_0(T), \forall T \in \mathcal{T}_h\}. 
\end{equation*}
 We let $P_0:L^2(\Omega) \rightarrow W_h$ be the $L^2$ orthogonal projection:
 \begin{equation*}
 (P_0 w, v_h)_{\Omega}=(w, v_h)_{\Omega} \quad \forall v_h \in W_h.  
 \end{equation*}

In order to stabilize a FEM we need the following bilinear form (see for example \cite{BE07})
\begin{equation}\label{eq:stab_form}
s(w_h,v_h) := \sum_{F \in \mathcal{F}} \int_F h_F^2 (|\beta \cdot n|+\varepsilon_{\beta^\perp})
\jump{\nabla w} \cdot \jump{\nabla v} ~\mbox{d}s,
\end{equation}
where we introduce the jump of the gradient
$$
\jump{\nabla w}\vert_F:=\lim_{\epsilon \rightarrow 0^+} \nabla w(x - \epsilon n_T) \cdot n_{T} +  \nabla w(x - \epsilon n_{T'}) \cdot n_{T'},
\mbox{ with } x\in F \quad \mbox{ and } F = \bar T \cap \bar T'
$$
and $\varepsilon_{\beta^\perp}\ge 0$ is a coefficient that
introduces some weakly consistent cross wind diffusion when non-zero.
We will also apply the jump to scalar quantities below in which case
it is defined by
\[
\jump{w}\vert_F = \lim_{\epsilon \rightarrow 0^+} w(x - \epsilon n_F)
- w(x+\epsilon n_F) \mbox{ with } x\in F
\]
where $n_F$ is a fixed but arbitrary normal to the face $F$.
We may then define the semi-norm
\begin{equation*}
|v|_s := s(v,v)^{\frac12}.
\end{equation*}

We let let $\beta_0$ be the Raviart-Thomas projection of $\beta$ to the lowest order Raviart-Thomas space. Since $\nabla \cdot \beta=0$ we have that $\beta_0$ is piecewise constant. Thus, we have 

\begin{alignat}{1}
\|\beta -\beta_0\|_{\infty} \le C h \|\nabla \beta\|_{\infty}.\label{beta0}
\end{alignat}

Then recall a critical approximation result that exhibits the
importance of the stabilization term, this follows from the local
estimate of \cite[Lemma 5.3]{BE07}:
\begin{equation}\label{interpinq}
\inf_{ v_h \in V_h} \| \beta_0 \cdot \nabla w_h - v_h\|^2 \leq C\,\sum_{F \in \mathcal{F}} h_F 
\|\jump{\beta_0 \cdot \nabla w_h} \|^2_F. 
\end{equation}
Note that since $\beta_0 \cdot n\vert_{\partial \Omega} = \beta \cdot
n\vert_{\partial \Omega}=0$ we have $\beta_0 \cdot \nabla
w_h\vert_{\partial \Omega} = 0$ and therefore \eqref{interpinq} holds taking the
infimum over the space $\mVh$, i.e.
\begin{equation}\label{interpinq2a}
\inf_{ v_h \in \mVh} \| \beta_0 \cdot \nabla w_h - v_h\|^2 \leq C\,\sum_{F \in \mathcal{F}} h_F 
\|\jump{\beta_0 \cdot \nabla w_h} \|^2_F. 
\end{equation}
Using \eqref{interpinq} together with \eqref{beta0} it is
straighforward to show that
\begin{equation}\label{interpinq2}
\inf_{ v_h \in \mVh} \|\beta \cdot \nabla w_h - v_h\| \leq
C\,(\|\nabla \beta\|_{\infty} \|w_h\|+ \left(\frac{\|\beta\|_{\infty}}{h}\right)^{\frac12}|w_h|_s)
\end{equation}

\flushleft
Indeed, we first add and subtract $\beta_0$ and apply the triangle inequality
\[
\|\beta \cdot \nabla w_h - v_h\| \leq \|(\beta - \beta_0)\cdot \nabla w_h\|+\| \beta_0 \cdot \nabla w_h- v_h\|.
\]
Then using \eqref{beta0}, an inverse inequality and \eqref{interpinq}
\[
\|(\beta - \beta_0)\cdot \nabla w_h\|+\| \beta_0 \cdot \nabla w_h-
v_h\| \leq C ( \|\nabla \beta\|_{\infty} \|w_h\| + (\sum_{F \in \mathcal{F}} h_F 
\|\jump{\beta_0 \cdot \nabla w_h} \|^2_F)^{\frac12}).
\]
Adding and subtracting $\beta$ in the second term and using a trace
inequality followed by \eqref{beta0},
\begin{alignat*}{1}
 \sum_{F \in \mathcal{F}} h_F 
\|\jump{\beta_0 \cdot \nabla w_h} \|^2_F \le & 
 C (\|(\beta -
\beta_0)\cdot \nabla w_h\|^2 +  \|\nabla \beta\|_{\infty}^2 \|w_h\|^2 +
 \sum_{F \in \mathcal{F}} h_F \|\jump{\beta \cdot \nabla w_h} \|_F^2) \\
\le & C ( \|\nabla \beta\|_{\infty}^2 \|w_h\|^2 +  \sum_{F \in \mathcal{F}} h_F \|\jump{\beta \cdot \nabla w_h} \|_F^2).
\end{alignat*}
Using the continuity of $w_h$ in the last term of the right hand side
we see that
\[
\|\jump{\beta \cdot \nabla w_h} \|_F = \|\beta \cdot n\jump{\nabla
  w_h} \|_F \leq \|\beta\|_\infty^{\frac12}  \||\beta \cdot n|^{\frac12}\jump{\nabla
  w_h} \|_F.
\]
Hence, we have shown \eqref{interpinq2}.

We can then defined the stabilised convection form
\[
c_h(w_h,v_h) := c(w_h,v_h) + \gamma s(w_h,v_h).
\]

Introducing $\tau$ as the time step size, we also define Courant
number $Co$ that will either be the standard hyperbolic CFL, $\Co :=
(\|\beta\|_{\infty}+1) \frac{\tau}{h}$,
or a slightly stronger $4/3$-Courant number (see \cite{ZS04, BEF10}
where it was used in the context of second order Runge-Kutta methods), $Co_{4/3} := 
\tau (\|\beta\|_\infty/h)^{\frac43}$, that will apply for
finite element spaces or polynomial degrees higher than $1$.
 Observe that $Co$ is a free parameter that can be made
as small as we like by making $\tau$ small relative to $h$ and
$\beta$. The crucial point is that certain time residual terms from
the convection term can be made as small as necessary by fixing $Co$
to be small enough. This is expressed in boundedness properties of the
convection and the associated stabilization that we now summarize. 
First note that by the skew symmetry of the convection we have the positivity
\begin{equation}\label{cgamma}
\gamma |v|_s^2 = c_h(v,v), \quad\forall v \in  H_0^1(\Omega)\cap
H^{\frac32+\epsilon}(\Omega) + \mVh,
\end{equation}
and by skew-symmetry followed by the Cauchy-Schwarz inequality, an
inverse inequality and the definition of $Co$ we have the positivity
\begin{equation}\label{eq:c_cont}
\tau c(v,w_h) \leq C_i Co \|v\| \|w_h\| \quad \forall v \in H_0^1(\Omega), w_h \in V_h,
\end{equation}
where $C_i$ the constant of an inverse inequality. 
Similarly for the stabilisation norm we have the bound
\begin{equation}\label{eq:s_cont}
\tau^{\frac12} |w_h|_s \leq C_i Co^{\frac12} \|w_h\| \quad \forall w_h
\in V_h.
\end{equation}
For the analysis we introduce a projection operator $C_h:H^1(\Omega)
\mapsto \mVh$ defined by
\[
(C_h v, v_h)_{\Omega} = c(v,v_h), \quad \forall v_h \in \mVh.
\]
Note that by the definition of $c$, $(C_h w_h, v_h)_{\Omega} = -(w_,C_h v_h)_{\Omega}$
for $w_h,v_h \in \mVh$. Using \eqref{eq:c_cont} it is straightforward to show that the $C_h$ operator satisfies the bound
\begin{equation}\label{eq:43cont}
\tau \|C_h v\| \leq C_i \tau^{\frac14} Co^{\frac34}_{4/3} \|v\|.
\end{equation}

\begin{proof} (inequality \eqref{eq:43cont})
\[
\tau^2 (C_h v, C_h v ) = \tau^2 c(v,C_h v) =  - \tau^2 c(C_h v,v) \leq
\tau^2 h^{-1} \|\beta\|_\infty C_i \|C_h v\| \|v\|.
\]
Therefore,
\[
\tau \|C_h v\| \leq \tau h^{-1} \|\beta\|_\infty C_i \|v\|,
\]
but $\tau =\tau^{\frac14} Co_{4/3}^{\frac34} h/\|\beta\|_\infty$ by which the claim follows.
\end{proof}

We also notice that 
\begin{equation}\label{ChCo}
\tau \|C_h v\| \leq C_i \Co \|v\|.
\end{equation}

\subsection{Operators for time discretization}
We define the second order backward differentiation operator
\begin{equation}\label{eq:bdf2}
D_{\tau} v^{n+1} := \frac{3 v^{n+1} - 4 v^n + v^{n-1}}{2 \tau}.
\end{equation}
We recall the second order extrapolation $\tilde v^{n+1} = 2 v^n -
v^{n-1}$, and the increment operator $\delta$ such that $\delta v^{n+1}
:= v^{n+1} - v^n$. Observe that there holds 
\begin{equation}\label{tildedelta}
\tilde v^{n+1}-v^{n+1} = \delta v^n - \delta v^{n+1} = -\delta \delta v^{n+1}.
\end{equation}
We also recall that
\begin{equation}\label{Dtau}
\tau D_{\tau} v^{n+1} = \delta v^{n+1}+ \frac{1}{2} \delta \delta v^{n+1}.
\end{equation}
Finally, we also observe that
\begin{equation}\label{diff_stab}
\|\delta^m v^n\| \leq 2 \sum_{i=0}^{m} \|v^{n-i}\|,
\quad m=1,2, \, \mbox{ and }2 \leq n \leq N.
\end{equation}

As we will describe in a later section, for the Crank-Nicolson method the approximation of the time derivative
is given by the scaled increment operator 
$\tau^{-1}  \delta v^{n+1}$. The extrapolation is taken to the time
level $t^{n+1/2}$, in order to approximate the central difference in
time that is the key feature of the Crank-Nicolson scheme, $\hat{v}^{n+1} = \frac32 v^n -
\frac12 v^{n-1}$.

For the time discretization part of the error analysis we need some well known results on truncation
error analysis of finite difference operators that we collect in the
following proposition for future reference. These results are
standard and can be found for instance in the monography \cite{Thom06}, but for
completeness we sketch the proofs.
\begin{proposition}
Let $u^{n} := u(t^n)$ and $y:= \beta\cdot \nabla u$ then there holds
\begin{equation}\label{eq:BDFtime_trunc}
\|D_\tau u^{n+1} - \partial_t u^{n+1}\|^2 \leq C \tau^3 \|u_{ttt}\|^2_{L^2(t^{n-1},t^{n+1};L^2(\Omega))};
\end{equation}
\begin{equation}\label{eq:BDFextra}
\|\beta(t^{n+1}) \cdot \nabla u^{n+1} - \tilde y^{n+1}\|^2 \leq C \tau^{3} \|(\beta \cdot \nabla u)_{tt}\|^2_{L^2(t^{n-1},t^{n+1};L^2(\Omega))}.
\end{equation}

\end{proposition}
\begin{proof}
We first consider the bound \eqref{eq:BDFtime_trunc}
We wish to bound
\[
\|D_\tau u^{n+1} - \partial_t u^{n+1}\|.
\]
Using Taylor development we may write
\[
u(t) = u^{n-1} + (t-t^{n-1}) \partial_t u^{n-1} + \frac12 (t-t^{n-1})^2 \partial^2_t
u^{n-1} + \underbrace{\frac12 \int_{t_{n-1}}^t (t-s)^2 u^{(3)}(s) \,ds}_{R(t)} = Q(t) + R(t).
\]
Deriving and evaluating at $t^{n+1}$ we see that
\[
Q'(t) = \partial_t u^{n-1} + (t-t^{n-1}) \partial^2_t
u^{n-1}, \quad Q'(t^{n+1}) = \partial_t u^{n-1} + 2 \tau \partial^2_t
u^{n-1}.
\]
We also have
\[
\tau^{-1} \delta Q^{n+1} =\partial_t u^{n-1} + \frac{1}{2\tau}
((t^{n+1}-t^{n-1})^2 - (t^{n}-t^{n-1})^2) \partial^2_t
u^{n-1} = \partial_t u^{n-1} + \frac{3 \tau}{2} \partial^2_t
u^{n-1}
\]
and
\begin{multline*}
\tau^{-1}\delta^2 Q^{n+1} = \underbrace{((t^{n+1}-t^{n-1}) - 2 (t^n
  - t^{n-1})) \partial_t u^{n-1}}_{=0} \\
+ \frac{1}{2\tau} ((t^{n+1}-t^{n-1})^2
- 2(t^{n}-t^{n-1})^2)  \partial^2_t
u^{n-1} = \tau^2 \partial^2_t
u^{n-1}.
\end{multline*}
Using \eqref{Dtau} we see that
\[
D_\tau Q^{n+1}  = \partial_t u^{n-1} +  \frac{3 \tau}{2} \partial^2_t
u^{n-1} + \frac{\tau}{2} \partial^2_t
u^{n-1}= \partial_t u^{n-1} +  2 \tau \partial^2_t
u^{n-1} = Q'(t^{n+1}).
\]
Therefore
\[
\|D_\tau u^{n+1} - \partial_t u^{n+1}\|^2 = \|\underbrace{D_\tau Q - Q' }_{=0}+ D_\tau R^{n+1}
- \partial_t R^{n+1}\|^2 \leq \frac{C}{\tau^2} \sum_{k=n-1}^{n+1}
\|R^k\|^2+2 \|\partial_t R^{n+1}\|^2.
\]
By the definition of $R$ and the Cauchy-Schwarz inequality:
\[
 \frac{1}{\tau^2} \sum_{k=n-1}^{n+1} \|R^k\|^2 \leq 
 C \tau^3 \int_{t^{n-1}}^{t^{n+1}} \|u^{(3)}\|^2 ~ds.
\]
Finally 
\[
\partial_t R^{n+1} = \int_{t^{n-1}}^{t^{n+1}} (t^n - s) u^{(3)}(s) ~ds 
\]
and therefore in a similar fashion
\[
 \|\partial_t R^{n+1}\|^2 \leq \tau^3 \int_{t^{n-1}}^{t^{n+1}} \|u^{(3)}\|^2 ~ds,
\]
which gives \eqref{eq:BDFtime_trunc}.
The result  \eqref{eq:BDFextra} easily follows after we apply the Cauchy-Shwarz inequality 
\begin{equation}\label{eq:T2}
\|y^{n+1} - \tilde y^{n+1}\|^2 = \|\int_{t^{n-1}}^{t^{n}} \int_{t}^{t + \tau} y_{tt}(s)
ds\, dt \|^2 \leq \tau^3 \int_{t_{n-1}}^{t^{n+1}} \|y_{tt}(s)\|^2 ds.
\end{equation}

\end{proof}
\section{The BDF2-IMEX Method}
We may write the BDF2-IMEX finite element method as follows.  
\flushleft
Find $u_h^{n+1} \in \mVh$ such that for $n \ge 1$,
\begin{equation}\label{eq:scheme}
(D_{\tau} u_h^{n+1},v_h)_\Omega +c_h(\tilde{u}_h^{n+1}, v_h)+a(u_h^{n+1}, v_h) 
= L^{n+1}(v_h),\quad \forall v_h \in \mVh,
\end{equation}
where $u_h^0, u_h^1$ are given. Here $\{L^n\}$ are a bounded linear operator on  $\mVh$.

\subsection{Stability of BDF2-IMEX}\label{sec:bdf2}
In the diffusion dominated ($Pe<1$) regime the BDF2-IMEX method is
stable under the standard hyperbolic CFL condition.
In this section we prove in addition to this, that BDF2-IMEX the method is stable
indepdendent of the Peclet number with a standard hyperbolic CFL
condition when $p=1$ and under the $4/3$-CFL when $p>1$. 

Let us define some norms. We start by defining the natural dissipation of the spatial variables.
\begin{equation*}
E(v)^2:=\gamma |v|_s^2+  \|\mu^{\frac12}
\nabla {v}\|^2.
\end{equation*}
We see that $E(v)^2=c_h(v, v)+a(v, v)$ when $v \in \mVh$ . An
immediate consequence of \eqref{eq:s_cont} and an inverse inequality
is that for all $v \in V_h$,
\begin{subequations}
\begin{alignat}{1}
\gamma |v|_s^2 \le & \frac{C \gamma (\|\beta\|_{\infty}+1)}{h} \|v\|^2,  \label{preEbound1} \\
   \|\mu^{\frac12} \nabla {v}\|^2 \le & \frac{C \mu}{h^2} \|v\|^2.  \label{preEbound2}
\end{alignat}
\end{subequations}
Hence, 
\begin{equation}\label{eq:Ebound}
\tau E(v)^2 \leq C \left(\gamma Co + \frac{Co}{Pe}\right) \|v\|^2,
\end{equation}
where we recall that the definition of the Peclet number $\Pe$:
\begin{equation*}
\Pe :=\frac{\|\beta\|_{\infty} h}{\mu}. 
\end{equation*}

For a linear operator $L$ defined for $\mVh$ we define
\begin{equation}\label{eq:Ldef}
\|L\|_h= \sup_{ v \in \mVh} \frac{L(v)}{\sqrt{E(v)^2+\|v\|^2}}. 
\end{equation}

 We introduce the triple norm, measuring the dissipation in the system,
\[
|||v|||^2 := \sum_{n=1}^{N-1} (\tau E(v^{n+1})^2 + \frac{1}{4}\|v^{n+1} - \tilde v^{n+1}\|^2).
\]
The following elementary relationship will be useful. 
\begin{equation}\label{aux104}
\tau (D_{\tau} v^{n+1},v^{n+1})_\Omega = \frac14 \big(
  \|v^{n+1}\|^2 + \|\tilde v^{n+2}\|^2 -(
  \|v^{n}\|^2 + \|\tilde v^{n+1}\|^2) + \|v^{n+1} - \tilde v^{n+1}\|^2 \big).
  \end{equation}

We will also make use of the following summation by parts formulas.
\begin{lemma}\label{lem:sum_by_parts}
Let $r(\cdot,\cdot)$ denote a bilinear form on $V_h \times V_h$. Then
the following summation by parts formulas holds
\begin{alignat}{1}
\sum_{n=1}^{N-1}  r(\delta v^{n+1}, w^{n+1})   =   r(v^{N}, w^{N}) - r(v^{1}, w^{2})  -
\sum_{n=2}^{N-1}  r(v^{n}, \delta w^{n+1}). \label{rbi1}
\end{alignat}
\end{lemma}
\begin{proof}
  We write
\begin{alignat*}{1}
\sum_{n=1}^{N-1}  r(\delta v^{n+1}, w^{n+1})  =& \sum_{n=1}^{N-1}  r(v^{n+1}, w^{n+1})- \sum_{n=1}^{N-1}  r(v^{n}, w^{n+1}) \\
= & r(v^N, w^N)+  \sum_{n=2}^{N-1}  r(v^{n},  w^{n})- \sum_{n=1}^{N-1}  r(v^{n}, w^{n+1}) \\
=& r(v^N, w^N)  - r(v^{1}, w^{2})  -\sum_{n=2}^{N-1}  r(v^{n}, \delta w^{n+1}).
\end{alignat*}

\end{proof}

\subsection{The case $\Pe \le 1$ for $p \ge 1$ and  the case $\Pe>1$ for $p=1$}
Before proving stability we prove an auxiliary result which will be helpful for the case $p=1$ and $\Pe >1$.  
\begin{lemma}
Let $p=1$ and let $u_h$ solve \eqref{eq:scheme} then the following estimate holds
\begin{alignat}{1}
\|\tau D_{\tau}u_h^{n+1}-P_0(\tau D_{\tau}u_h^{n+1})\| \le& C \sqrt{\tau} \sqrt{\Co} K E(u_h^{n+1}) +C \tau \|\nabla \beta\|_{\infty} \|\tilde{u}_h^{n+1}\| +C \Co \|\tilde{u}_h^{n+1}-u_h^{n+1}\| \nonumber\\
&+C\sqrt{\tau}  ( \frac{\sqrt{\Co}}{\sqrt{\Pe}}+ \sqrt{\Co}\sqrt{\gamma}+1) \|L^{n+1}\|_h, \label{DtauU}
\end{alignat}
where
\begin{equation*}
K(\gamma,  \Pe):= \Big(\frac{1}{\Pe}+\sqrt{\gamma}+\frac{1}{\sqrt{\gamma}}\Big).
\end{equation*}
\end{lemma}

\begin{proof}
Let $y_h=\tau D_{\tau}u_h^{n+1}$ and then we have by \eqref{eq:scheme}
\begin{alignat*}{1}
\|y_h-P_0 y_h\|^2=&( y_h, y_h-P_0(y_h))_{\Omega}\\
=&(y_h, \pi_h( y_h-P_0(y_h)))_{\Omega}\\
=& -\tau c_h(\tilde u_h^{n+1}, \pi_h( y_h-P_0(y_h)))-\tau a( u_h^{n+1}, \pi_h( y_h-P_0(y_h))) \\
&+\tau L^{n+1}(\pi_h( y_h-P_0(y_h))).
\end{alignat*}
We use the Cauchy-Schwarz inequality followed inverse estimates to bound the symmetric terms
\begin{alignat*}{1}
- \tau  a( u_h^{n+1}, \pi_h( y_h-P_0(y_h)))\le& C \frac{\sqrt{\mu}\tau}{h} \|\sqrt{\mu}  \nabla u_h^{n+1}\| \, \| y_h-P_0(y_h)\| \\
 \le& C \frac{\sqrt{\tau \Co}}{\sqrt{\Pe}} \|\sqrt{\mu}  \nabla u_h^{n+1}\| \, \| y_h-P_0(y_h)\|.
\end{alignat*}
For the stabilization we apply \eqref{eq:s_cont}
\begin{alignat}{1}
-\tau   \gamma s(\tilde u_h^{n+1},\pi_h( y_h-P_0(y_h))) 
\le&   C\sqrt{\tau} \sqrt{\gamma \Co}         \sqrt{\gamma} | \tilde
u_h^{n+1}|_s \| y_h-P_0(y_h)\|. \label{eq:stab_cont}
\end{alignat}

Next we bound $\tau L^{n+1}(\pi_h( y_h-P_0(y_h)))$ using \eqref{eq:Ldef} and \eqref{eq:Ebound}.
\begin{alignat*}{1}
\tau L^{n+1}(\pi_h( y_h-P_0(y_h))) \le &{\tau}\|L^{n+1}\|_h \sqrt{ E( \pi_h( y_h-P_0(y_h)))^2+ \|\pi_h( y_h-P_0(y_h))\|^2} \\
 \le &C\sqrt{\tau} ( \frac{\sqrt{\Co}}{\sqrt{\Pe}}+ \sqrt{\gamma \Co} +\sqrt{\tau}) \|L^{n+1}\|_h   \|y_h-P_0(y_h)\|.
\end{alignat*}

To bound the first term we observe that by \eqref{eq:stab_cont} it
only remains to bound the contribution from the form $c$.
\begin{alignat*}{1}
-\tau c(\tilde u_h^{n+1}, \pi_h( y_h-P_0(y_h))) =&-\tau ((\beta-\beta_0) \cdot \nabla \tilde u_h^{n+1}, \pi_h( y_h-P_0(y_h)))_{\Omega} \\
&-\tau (\beta_0 \cdot \nabla \tilde u_h^{n+1}-w_h^{n+1}, (I-\pi_h)( y_h-P_0(y_h)))_{\Omega}.
\end{alignat*}
Here $w_h \in \mVh$ is arbitrary. Note that we crucially used that $p=1$ which implies that $\beta_0 \cdot \nabla \tilde u_h^{n+1} \in W_h$.

Hence, using \eqref{beta0} and \eqref{interpinq} we obtain
\begin{alignat*}{1}
&-\tau c(\tilde u_h^{n+1}, \pi_h( y_h-P_0(y_h)))\\
 \le & C \tau (\|\nabla \beta\|_{\infty}\|\tilde u_h^{n+1}\|+ \|\beta\|_{\infty}^{1/2} h^{-1/2} | \tilde u_h^{n+1}|_s)  \|y_h-P_0(y_h)\| \\
\le & C \sqrt{\tau} (\sqrt{\tau} \|\nabla \beta\|_{\infty} \|\tilde u_h^{n+1}\|+ \sqrt{\Co} \frac{1}{\sqrt{\gamma}}  \sqrt{\gamma}| \tilde u_h^{n+1}|_s)  \|y_h-P_0(y_h)\|.
\end{alignat*}

Finally, by the triangle inequality and \eqref{eq:s_cont} we have the bound $| \tilde u_h^{n+1}|_s \le | u_h^{n+1}|_s+ \frac{C \sqrt{\Co}}{\sqrt{\tau}} \|u_h^{n+1}-\tilde u_h^{n+1}\|$.  Combining the above inequalities gives the result.
\end{proof}

We will need the following discrete simple form of the discrete Gronwall's inequality. 
\begin{proposition}
Let $\{\phi_n\}$ be a sequence of non-negative numbers and let $\psi$ and $\eta$ be non-negative numbers such that
\begin{equation*}
\phi_n \le \psi+ \eta \sum_{i=1}^n \phi_i
\end{equation*}
Then, the following estimate holds
\begin{equation}\label{gronwall}
\phi_N \le (1+ N \eta e^{\eta N}) \psi. 
\end{equation}

\end{proposition}

We will use the following notation: We let $c_L=0$ if $L^{n}\equiv 0$ for every $n$ and $c_L=1$ otherwise.

\begin{theorem}\label{thm:stab}
Suppose that  $T=N \tau$. Suppose that $\Co$ is chosen sufficiently small only depending on geometric constants of the mesh and $\gamma$.  For $\{u_h^n\}$ solving \eqref{eq:scheme} we have the following bounds:

If $\Pe \le 1$ then for all $p \ge 1$,
\begin{alignat}{1}
\|u_h^{N}\|^2   \le &   (1+\frac{Tc_L}{8} e^{\frac{Tc_L}{8}}) (\|u_h^{0}\|^2 +\|u_h^{1}\|^2 +32 \tau \sum_{n=1}^{N-1} \|L^{n+1}\|_h^2 )  . \label{thmresult1}
\end{alignat}

If $\Pe > 1$  and $p=1$, $\gamma>0$ then 
\begin{alignat}{1}
\|u_h^{N}\|^2  \le &    (\|u_h^{0}\|^2 +\|u_h^{1}\|^2+ \|u_h^{2}\|^2+ 9 \tau \sum_{n=1}^{N-1} \|L^{n+1}\|_h^2) M,  \label{thmresult2}
\end{alignat}
where 
\begin{equation*}
M=8\Big(1+c T (\tau \|\nabla \beta\|_{\infty}^2+c_L) e^{c T (\tau \|\nabla \beta\|_{\infty}^2+c_L)}\Big). 
\end{equation*}
\end{theorem}
\begin{proof}
We test equation \eqref{eq:scheme} with $u_h^{n+1}$ and get
\begin{equation*}
(D_{\tau} u_h^{n+1},u_h^{n+1})_\Omega +c_h(\tilde u_h^{n+1}, u_h^{n+1}) +a(u_h^{n+1}, u_h^{n+1})
= L^{n+1}(u_h^{n+1}).
\end{equation*}
Thus, we see that 
\begin{equation*}
(D_{\tau} u_h^{n+1},u_h^{n+1})_\Omega +c_h(u_h^{n+1}, u_h^{n+1})+a(u_h^{n+1}, u_h^{n+1})
=c_h(u_h^{n+1}-\tilde u_h^{n+1},  u_h^{n+1})+  L^{n+1}(u_h^{n+1}). 
\end{equation*}

Then by summing over $n=1,\hdots, N-1$, multiplying with $\tau$, and using \eqref{aux104}, \eqref{cgamma}
\begin{alignat}{1}
&\frac14 (
  \|u_h^{N}\|^2 + \|\tilde u_h^{N+1}\|^2) -\frac14(
  \|u_h^{1}\|^2 + \|\tilde u_h^{2}\|^2) + |||u_h|||^2 = S_1+S_2+S_3 \label{aux436}
\end{alignat}
where
\begin{alignat*}{1} 
S_1:=& \tau \sum_{n=1}^{N-1} c(u_h^{n+1}-\tilde u_h^{n+1},  u_h^{n+1}), \\
S_2:=&\tau \sum_{n=1}^{N-1} \gamma s(u_h^{n+1}-\tilde u_h^{n+1},  u_h^{n+1}),\\
S_3:=& \tau \sum_{n=1}^{N-1} L^{n+1}(u_h^{n+1}).\\
\end{alignat*}
Let us estimate $S_3$. We have 
\begin{alignat*}{1}
S_3 \le & \tau \sum_{n=1}^{N-1} \|L^{n+1}\|_h \sqrt{E(u_h^{n+1})^2+\|u_h^{n+1}\|^2} \\
\le &  8 \tau \sum_{n=1}^{N-1} \|L^{n+1}\|_h^2 + \frac{1}{32} \tau  \sum_{n=1}^{N-1} E(u_h^{n+1})^2+ \frac{c_L}{32} \tau  \sum_{n=1}^{N-1} \|u_h^{n+1}\|^2. 
\end{alignat*}

Now we estimate $S_2$.  Using the arithmetic-geometric mean inequality and inverse estimates we obtain
\begin{equation*}
S_2 \le  \sqrt{\Co}   \gamma \sum_{n=1}^{N-1}   \| u_h^{n+1}-\tilde
u_h^{n+1}\|^2 +   C \sqrt{\Co} \tau  \sum_{n=1}^{N-1}  \gamma  | u_h^{n+1}|_s^2.
\end{equation*}

Next we bound $S_1$. We consider two cases: $\Pe >1$ and $\Pe  \le 1$.

{{\bf Case 1}: $\Pe \le 1$, $p \ge 1$:} Using that $c(u_h^{n+1}-\tilde u_h^{n+1},  u_h^{n+1})=-c(u_h^{n+1}, u_h^{n+1}-\tilde u_h^{n+1})$ we obtain 
\begin{alignat*}{1}
S_1   \le &    C \sqrt{\Co} \sqrt{\Pe} \sqrt{\tau} \sum_{n=1}^{N-1} \| u_h^{n+1}-\tilde u_h^{n+1}\|  \, \|\sqrt{\mu} \nabla u_h^{n+1}\| \\
\le & \sqrt{\Co}   \sum_{n=1}^{N-1} \| u_h^{n+1}-\tilde u_h^{n+1}\|^2+   C \sqrt{\Co}  \tau \sum_{n=1}^{N-1} \|\sqrt{\mu} \nabla u_h^{n+1}\|^2.  
\end{alignat*}
Thus, using \eqref{aux436} and the fact that $\Co$  is sufficiently small we obtain
\begin{alignat*}{1}
&\frac14 (
  \|u_h^{N}\|^2 + \|\tilde u_h^{N+1}\|^2)  + \frac{1}{2} |||u_h|||^2  \\
&  \le \frac14(\|u_h^{1}\|^2 + \|\tilde u_h^{2}\|^2)+ 8 \tau \sum_{n=1}^{N-1} \|L^{n+1}\|_h^2   +\frac{c_L}{32} \tau  \sum_{n=1}^{N-1} \|u_h^{n+1}\|^2. 
\end{alignat*}
Using Gronwall's inequality \eqref{gronwall} we have \eqref{thmresult1}.

{{\bf Case 2}: $\Pe > 1$ and $p=1$:}
We use \eqref{tildedelta} and \eqref{rbi1} to obtain
\begin{alignat*}{1}
 S_1=& \tau \sum_{n=1}^{N-1} c(\delta \delta u_h^{n+1},  u_h^{n+1}) =
 - \tau \sum_{n=2}^{N-1} c(\delta u_h^{n},  \delta u_h^{n+1})+\tau
 \big(c(\delta u_h^N, u_h^N)-c(\delta u_h^1, u_h^2)\big). 
 \end{alignat*}
 Using that $c(\delta \delta u_h^{n+1},\delta \delta u_h^{n+1} )=0=c(\delta u_h^{n+1}, \delta u_h^{n+1} ) $ and \eqref{Dtau} we have
 \begin{equation*}
- c(\delta u_h^{n},  \delta u_h^{n+1})=c(\delta \delta u_h^{n+1},   \delta u_h^{n+1}) =c(\delta \delta u_h^{n+1}, \tau D_{\tau} u_h^{n+1}).
 \end{equation*}
 Thus, 
\begin{alignat}{1}
 S_1=&  \tau \sum_{n=2}^{N-1} c(\delta \delta u_h^{n+1}, \tau D_{\tau}
 u_h^{n+1})+\tau \big(c(\delta u_h^N, u_h^N)-c(\delta u_h^1,
 u_h^2)\big). \label{eq:crucial_S1}
 \end{alignat} 
We let $y_h^{n+1}= \tau D_{\tau} u_h^{n+1}$ and use the fact that
$P_0(y_h^{n+1})$ is in the kernel of the gradient operator followed by
an inequality similar to 
\eqref{eq:c_cont}, but applied elementwise, to see that 
\begin{alignat*}{1}
\tau c(\delta \delta u_h^{n+1},  y_h^{n+1})=-\tau c( y_h^{n+1},  \delta \delta u_h^{n+1})
\le & C Co \|\delta \delta u_h^{n+1}\|  \|y_h^{n+1}-P_0(y_h^{n+1})\|.
\end{alignat*}
Thus, applying  \eqref{DtauU} we obtain 
\begin{alignat*}{1}
 &\tau \sum_{n=2}^{N-1} c(\delta \delta u_h^{n+1},  \tau D_{\tau} u_h^{n+1})  \\
 \le & C (\Co)^{3/2}K\sqrt{\tau} \sum_{n=2}^{N-1}  E(u_h^{n+1})  \|\delta \delta u_h^{n+1}\| + C \Co \|\nabla \beta\|_{\infty}  \tau \sum_{n=2}^{N-1}   \|\tilde u_h^{n+1}\|\, \|\delta \delta u_h^{n+1}\| \\ 
 &+   C (\Co)^{2} \sum_{n=2}^{N-1} \|\delta \delta u_h^{n+1}\|^2 +C \sqrt{\tau }  \Co ( \frac{\sqrt{\Co}}{\sqrt{\Pe}}+ \sqrt{\Co \gamma}+\sqrt{\tau})  \sum_{n=2}^{N-1} \|L^{n+1}\|_h \|\delta \delta u_h^{n+1}\|.
\end{alignat*}
To bound the remaining two terms we use \eqref{eq:c_cont} followed by
Young's inequality:
 \begin{alignat}{1}
 &\tau \big(c(\delta u_h^N, u_h^N)-c(\delta u_h^1, u_h^2)\big)
 \le C \Co (\|u_h^N\|^2+\|\tilde u_h^{N+1}\|^2)+  C \Co (\|u_h^0\|^2+\|u_h^{1}\|^2+\|u_h^2\|^2). \label{S1part2}
 \end{alignat}
 Here we also used that $\delta u_h^N=\tilde{u}_h^{N+1}-u_h^N$. Hence, we arrive at 
\begin{alignat}{1}\nonumber
S_1 \le&   C \Co (\|u_h^N\|^2+\|\tilde{u}_h^{N+1}\|^2)+  C \Co (\|u_h^0\|^2+\|u_h^{1}\|^2+\|u_h^2\|^2) \\\nonumber
&+C  \Co K^2  \sum_{n=2}^{N-1} \tau E (u_h^{n+1})^2 + C (\Co)^{2} \sum_{n=2}^{N-1} \|\delta \delta u_h^{n+1}\|^2\\
  &+ C  \|\nabla \beta\|_{\infty}^2  \tau^2 \sum_{n=2}^{N-1}     \|
  \tilde{u}_h^{n+1}\|^2 +\tau     \sum_{n=2}^{N-1} \|L^{n+1}\|_h^2. \label{eq:S1bound}
\end{alignat}
Here we used that $( \frac{\sqrt{\Co}}{\sqrt{\Pe}}+ \sqrt{\Co \gamma}+\sqrt{\tau}) $ is bounded in this case. 
Finally, using \eqref{aux436} and the fact that $\Co$ is sufficiently small  we obtain
\begin{alignat}{1}
&\frac18 (
  \|u_h^{N}\|^2 + \|\tilde u_h^{N+1}\|^2)  + \frac{1}{2}|||u_h|||^2  \nonumber  \\
\le &   (\|u_h^{0}\|^2 +\|u_h^{1}\|^2 + \|u_h^{2}\|^2) + C   \tau(\tau \|\nabla \beta\|_{\infty}^2+c_L) \sum_{n=1}^{N-1}   \|u_h^{n+1}\|^2+9\tau  \sum_{n=1}^{N-1} \|L^{n+1}\|_h^2 \label{aux991}
\end{alignat}
We can now use the discrete Gronwall inequality \eqref{gronwall} to get \eqref{thmresult2}.
\end{proof}

\subsection{The case $ \Pe >1 $ for $p \ge 1$ with  4/3-CFL
condition}
We will now prove a stability result in the high Peclet regime $Pe>1$
that holds for any polynomial order under more stringent 4/3-CFL
condition. In fact, we will not need the stabilization term $s(\cdot,
\cdot)$ to guarantee this. The result holds for the standard Galerkin
method as well.
\begin{theorem}\label{thm:stab2}
Suppose that $T=N \tau$, $\Pe > 1$  and that $\max\{\Co, \Co_{4/3}\}$ is sufficiently small   only depending on geometric constants of the mesh and $\gamma$.  Let $p\ge 1$.  For $\{u_h^n\}$ solving \eqref{eq:scheme} we have the following bound:
\begin{alignat}{1}
\|u_h^{N}\|^2  \le &    (\|u_h^{0}\|^2 +\|u_h^{1}\|^2+ \|u_h^{2}\|^2+ 9\tau \sum_{n=1}^{N-1} \|L^{n+1}\|_h^2) M,  \label{thmresult4}
\end{alignat}
where 
\begin{equation*}
M= 8(1+(\frac{c_L}{4}+8) T e^{(\frac{c_L}{4}+8) T}).
\end{equation*}
\end{theorem}
\begin{proof}

Using the previous proof we only have to bound $S_1$ in the case, $\Pe \ge 1$ and $p > 1$. We will use \eqref{eq:crucial_S1} and in particular use the same estimate \eqref{S1part2}. We are left to bound $\tau \sum_{n=2}^{N-1} c(\delta \delta u_h^{n+1}, \tau D_{\tau}
 u_h^{n+1})$. To this end, we set $\psi_h= C_h \delta \delta u_h^{n+1}$ and see that 
\[
c(\delta \delta u_h^{n+1}, \tau D_{\tau}
 u_h^{n+1}) = \tau  (D_{\tau} u_h^{n+1}, \psi_h)_\Omega.
\]
Using the definition of the method \eqref{eq:scheme} it follows that
\[
\tau^2  (D_{\tau} u_h^{n+1}, \psi_h)_\Omega =M_1+M_2+M_3+M_4.
\]
where
\begin{alignat*}{2}
M_1:&= -\tau^2 c(\tilde u_h^{n+1},  \psi_h), \quad &&M_2 :=-\tau^2 \gamma s(\tilde u_h^{n+1},  \psi_h), \\
M_3:&= - \tau^2  a(u_h^{n+1},  \psi_h), \quad  && 
M_4:=\tau^2 L^{n+1} (\psi_h).
\end{alignat*}
We use again the definition of $C_h$ and the estimate \eqref{eq:43cont} to obtain
\begin{alignat*}{1}
M_1=-\tau^2(C_h \tilde u_h^{n+1}, \psi_h)_{\Omega} \le  \tau^2 \|C_h \tilde u_h^{n+1}\| \|\psi_h\| \le C \sqrt{\tau} (\Co_{4/3})^{3/2} \| \tilde u_h^{n+1}\| \|\delta \delta u_h^{n+1}\|.
\end{alignat*}
Using \eqref{preEbound1} and \eqref{ChCo} we have
\begin{alignat*}{1}
M_2 \le \tau^2 \gamma |\tilde u_h^{n+1}|_s  |\psi_h|_s \le   \tau^{3/2} \sqrt{\gamma} |\tilde u_h^{n+1}|_s  \sqrt{\gamma \Co}  \|\psi_h\| \le \sqrt{\tau} \sqrt{\gamma} |\tilde u_h^{n+1}|_s   \sqrt{\gamma} \Co^{3/2} \|\delta \delta u_h^{n+1}\|. 
\end{alignat*}
Similarly, we bound $M_3$ if we use \eqref{preEbound2}
\begin{alignat*}{1}
M_3 \le \tau^2  \|\sqrt{\mu} \nabla u_h^{n+1}\|  \| \sqrt{\mu} \nabla \psi_h\| \le  C \tau^2 \|\sqrt{\mu} \nabla u_h^{n+1}\|   \frac{\sqrt{\mu}}{h}  \|\psi_h\| \le   C \sqrt{\tau} \|\sqrt{\mu} \nabla u_h^{n+1}\| \ \frac{\Co^{3/2}}{\sqrt{\Pe}} \|\delta \delta u_h^{n+1}\|.
\end{alignat*}

Finally, using \eqref{eq:Ebound}
\begin{alignat*}{1}
M_4 \le& \tau^2 \|L_h^{n+1}\|_h \sqrt{E(\psi_h)^2+ \|\psi_h\|^2} \\
\le &C \sqrt{\tau}   \|L_h^{n+1}\|_h   \left(\sqrt{\gamma Co} + \frac{\sqrt{Co}}{\sqrt{Pe}}+ \sqrt{\tau}\right) \tau   \|\psi_h\|.
\end{alignat*}
Hence, applying \eqref{ChCo} we get 
\begin{alignat*}{1}
M_4 \le \tau^2 \|L_h^{n+1}\|_h \sqrt{E(\psi_h)^2+ \|\psi_h\|^2} \le C \sqrt{\tau}   \|L_h^{n+1}\|_h   \Co  \|\delta \delta u_h^{n+1}\|,
\end{alignat*}
where we use that  $\left(\gamma Co + \frac{Co}{Pe}+ \sqrt{\tau}\right)$ is bounded.

Hence, after using Young's inequality, \eqref{S1part2} and the fac that $\Pe > 1$ we arrive at
\begin{alignat*}{1}\nonumber
S_1 \le&   C \Co (\|u_h^N\|^2+\|\tilde{u}_h^{N+1}\|^2)+  C \Co (\|u_h^0\|^2+\|u_h^{1}\|^2+\|u_h^2\|^2) \\\nonumber
&+ \Co \sum_{n=2}^{N-1} \tau E (u_h^{n+1})^2 + C \big(\Co^{2} + (\Co_{4/3})^3\big) \sum_{n=2}^{N-1} \|\delta \delta u_h^{n+1}\|^2\\
  &+ \tau \sum_{n=2}^{N-1}     \|
  u_h^{n+1}\|^2 +\tau     \sum_{n=2}^{N-1} \|L^{n+1}\|_h^2.
\end{alignat*}
We now use \eqref{aux436} and the estimates $S_2$ and $S_3$ from the previous theorem. In addition, we use that $\Co$ and $\Co_{4/3}$ is sufficiently small to obtain
\begin{alignat*}{1}
&\frac18 (
  \|u_h^{N}\|^2 + \|\tilde u_h^{N+1}\|^2)  + \frac{1}{2} |||u_h|||^2  \\
&  \le (\|u_h^{0}\|^2 + \| u_h^1\|^2+\|u_h^2\|^2 )+ 9 \tau \sum_{n=1}^{N-1} \|L^{n+1}\|_h^2   +(\frac{c_L}{32}+1) \tau  \sum_{n=1}^{N-1} \|u_h^{n+1}\|^2. 
\end{alignat*}
The  estimate \eqref{thmresult4} now follows from Gronwalls inequality \eqref{gronwall}.
\end{proof}

We notice that $\|u_h^2\|$ appear in the right-hand side of some of the estimates; see for example \eqref{thmresult4}. However, we we can easily show (we omit the details) that if $\Co$ sufficiently small  
\begin{equation*}
\|u_h^2\| \le 2( \|u_h^{0}\|^2 +\|u_h^{1}\|^2)+ C\tau \|L^{2}\|_h^2.
\end{equation*}

If we combine Theorems \ref{thm:stab} and \ref{thm:stab2}  with this last inequality we get. 

\begin{corollary}
Let  $T=N \tau$. Let $\{u_h^n\}$ solving \eqref{eq:scheme} we have the following bounds:

If $\Pe \le 1$ and  for all $p \ge 1$, if $\Co$ is sufficiently small we have :
\begin{alignat}{1}
\|u_h^{N}\|^2   \le &   (1+\frac{Tc_L}{8} e^{\frac{Tc_L}{8}}) (\|u_h^{0}\|^2 +\|u_h^{1}\|^2 +32 \tau \sum_{n=1}^{N-1} \|L^{n+1}\|_h^2 )  . \label{cor1}
\end{alignat}

If $\Pe > 1$,  $p=1$, $\gamma>0$  and $\Co$ sufficiently small we obtain:
\begin{alignat}{1}
\|u_h^{N}\|^2  \le &     C\big(\|u_h^{0}\|^2 +\|u_h^{1}\|^2+ \tau \sum_{n=1}^{N-1} \|L^{n+1}\|_h^2\big) M,  \label{cor2}
\end{alignat}
where 
\begin{equation}\label{Mhere}
M=8\Big(1+c T (\tau \|\nabla \beta\|_{\infty}^2+c_L) e^{c T (\tau \|\nabla \beta\|_{\infty}^2+c_L)}\Big). 
\end{equation}

If $\Pe > 1$,  $p  \ge 1$,  and $\max \{\Co, \Co_{4/3}\}$ is  sufficiently small we get:
\begin{alignat}{1}
\|u_h^{N}\|^2  \le &    C (\|u_h^{0}\|^2 +\|u_h^{1}\|^2+ \tau \sum_{n=1}^{N-1} \|L^{n+1}\|_h^2) M,  \label{cor3}
\end{alignat}
where $M$ is given in \eqref{Mhere}.
\end{corollary}

\begin{corollary}
Let  $T=N \tau$. Let $\{u_h^n\}$ solving \eqref{eq:scheme}.
If $\Pe > 1$,  $p=1$, $\gamma>0$  and $\Co$ sufficiently small we obtain:
\begin{alignat}{1}
||||u_h|||^2  \le &     C \big(1+ T (\tau \|\nabla \beta\|_{\infty}^2+c_L)\big) \big(\|u_h^{0}\|^2 +\|u_h^{1}\|^2+ \tau \sum_{n=1}^{N-1} \|L^{n+1}\|_h^2\big) M,  \label{corformat}
\end{alignat}
and $M$ is given in \eqref{Mhere}. 
\end{corollary}

\section{A priori error estimate for BDF2-IMEX in the case $\Pe  >1$}\label{sec:apriori}
In this section we will study the error in the BDF2-IMEX method. We focus on the case $\Pe>1$ for simplicity.  If $u$ solves \eqref{weaku}, its approximation is given by:

Find $u_h^{n+1} \in \mVh$ such that for $n \ge 1$,

\begin{equation}\label{BDF2weak}
(D_{\tau} u_h^{n+1},v_h)_\Omega +c_h(\tilde u_h^{n+1}, v_h)+a(u_h^{n+1}, v_h) 
= (f, v_h)_{\Omega},\quad \forall v_h \in \mVh,
\end{equation}
with $u_h^i = \pi_h u(\cdot,t_i)$, $i=0,1$.

We can prove an error estimate in the case  of $\Pe  > 1$ (i.e. convection dominated regime).  
\begin{theorem}\label{thmerrorestimate}
Let $u$ be the solution of \eqref{weaku} and $\{u_h^n\}_{n=0}^N$ be the solution of
\eqref{BDF2weak}. Let $T=\tau N$  and assume that $\gamma >0$. Furthermore, suppose that $\Co$ is sufficiently small when $p=1$ and $\max\{\Co, \Co_{4/3}\}$ is sufficiently small when $p \ge 2$. Then, 
\begin{alignat}{1}
\|  \pi_h u(T)- u_h^N \|^2 \le  C(h^{2(p+1)} (\|u(t_0)\|_{H^{p+1}(\Omega)}^2+\|u(t_1)\|_{H^{p+1}(\Omega)}^2)+ \mathsf{G}) \mathsf{M} \label{mainerror}
\end{alignat}
where
\begin{alignat*}{1}
\mathsf{G}:=& \big(\tau^4 \|\partial_t^3 u\|_{L^2([0,T], L^2(\Omega))}^2+ \|\beta\|_{\infty}^2  \tau^4 \|\partial_t^2  u\|_{L^2([0,T], H^1(\Omega))}^2\big) \\
&+  (\mu  h^{2p}+\gamma  \|\beta\|_{\infty}  h^{2(p+1/2)} +\|\nabla \beta \|_{\infty}^2  h^{2(p+1)} ) \tau \sum_{j=1}^{N-1} \|u^{n+1}\|_{H^{p+1}(\Omega)}^2.
\end{alignat*}
and
\begin{equation*}
\mathsf{M}=\Big(1+c T (\tau \|\nabla \beta\|_{\infty}^2+1) e^{c T (\tau \|\nabla \beta\|_{\infty}^2+1)}\Big). 
\end{equation*}
\end{theorem}
\begin{proof}
Let $w_h^n=\pi_h u^n$ and let $e_h^n =w_h^n  - u_h^n$ where $\pi_h$ is
defined by \eqref{eq:L2proj}. Moreover, we let $\eta_h^n=w_h^n-u^n$. Then, we have that 
\begin{equation}\label{eherroreq}
(D_{\tau} e_h^{n+1},v_h)_\Omega +c_h(\tilde e_h^{n+1}, v_h)+a(e_h^{n+1}, v_h) 
= L^{n+1}(v_h),\quad \forall v_h \in \mVh,
\end{equation}
where 
\begin{equation*}
 L^{n+1}(v_h):= (D_{\tau} w_h^{n+1},v_h)_\Omega +c_h(\tilde w_h^{n+1}, v_h)+a(w_h^{n+1}, v_h) 
 -  \Big((\partial_t u^{n+1},v_h)_\Omega +c(u^{n+1}, v_h)+a(u^{n+1}, v_h)\Big).
\end{equation*}
We can write
\begin{alignat*}{1}
 L^{n+1}(v_h):=& \sum_{j=1}^5 \Psi_j(v_h),
 \end{alignat*}
 where 
\begin{alignat*}{2}
\Psi_1(v_h):=&(D_{\tau} u^{n+1}-\partial_t u^{n+1}, v_h)_{\Omega}, \quad &&  \Psi_2(v_h):= c(\tilde{\eta}_h^{n+1}, v_h),  \\
\Psi_3(v_h):=& \gamma s(\tilde{\eta}_h^{n+1}, v_h), \quad && \Psi_4(v_h):= a(\eta_h^{n+1}, v_h),  \\
 \Psi_5(v_h):=&   c(\tilde{u}^{n+1}-u^{n+1}, v_h). \quad &&
\end{alignat*}
All the terms above can easily be bounded. However, we have to pay
special care to  $\Psi_2(v_h)$. Using the skew-symmetry of $c$, and
the $L^2$-orthogonality of
$\tilde{\eta}_h^{n+1}$ we can subtract an arbitrary $z_h \in \mVh$ from
the convective derivative. Then by the Cauchy-Schwarz inequality we
see that 
\begin{alignat*}{1}
 \Psi_2(v_h)=&-c(v_h, \tilde{\eta}_h^{n+1})=-(\beta \cdot \nabla v_h -
 z_h,  \tilde{\eta}_h^{n+1})_{\Omega} 
\leq  \inf_{z_h \in \mVh} \|\beta \cdot \nabla v_h-z_h\| \|\tilde{\eta}_h^{n+1}\| .
\end{alignat*}
Hence,  using  \eqref{interpinq2} we obtain
 \begin{alignat*}{1}
 |\Psi_2(v_h)| 
 \le &C(\|\nabla \beta\|_{\infty} \|v_h\|+ \frac{\sqrt{\|\beta\|_{\infty}}}{\sqrt{h} \sqrt{\gamma}} \sqrt{\gamma}|v_h|_s)  \|\tilde{\eta}_h^{n+1}\|.
\end{alignat*}
As a consequence of this bound for $\Psi_2$ and by bounding all the
other terms $\Psi_i$ using the Cauchy-Schwarz inequality we have
\begin{alignat*}{1}
\|L^{n+1}\|_h \le & \|D_{\tau} u^{n+1}-\partial_t u^{n+1}\|+ \|\beta \cdot  \nabla (\tilde{u}^{n+1}-u^{n+1})\|\\
&+C(\|\nabla \beta\|_{\infty} + \frac{\sqrt{\|\beta\|_{\infty}}}{\sqrt{h} \sqrt{\gamma}}) \|\tilde{\eta}_h^{n+1}\|+C \sqrt{\gamma} | \tilde{\eta}_h^{n+1}|_s+ C \| \sqrt{\mu} \nabla \eta_h^{n+1}\|_{L^2(\Omega)}. 
\end{alignat*}
Using  \eqref{eq:BDFtime_trunc},  \eqref{eq:BDFextra} and  \eqref{eq:BDFextra}, the square of the first four terms of the right hand side can be bounded as follows. 
\begin{alignat*}{1}
\|D_{\tau} u^{n+1}-\partial_t u^{n+1}\|^2 \le &  C \tau^3 \int_{t_n}^{t_{n+1}} \|\partial_t^3 u(\cdot, s)\|^2 ds, \quad \\
\|\beta \cdot  \nabla (\tilde{u}^{n+1}-u^{n+1})\|^2 \le & C
\|\beta\|_{\infty}^2  \tau^3 \|\partial_t^2 \nabla u
\|^2_{L^2(t_{n-1},t_{n+1};L^2(\Omega))} , \quad \\
\| \tilde{\eta}_h^{n+1}\|^2 \le & C h^{2(p+1)} \|\tilde{u}^{n+1}\|_{H^{p+1}(\Omega)}^2,  \\
 \gamma | \tilde{\eta}_h^{n+1}|_s^2  \le & C \gamma \|\beta\|_{\infty} h^{2(p+1/2)} \|\tilde{u}^{n+1}\|_{H^{p+1}(\Omega)}^2 \\
  \| \sqrt{\mu} \nabla \eta_h^{n+1}\|_{L^2(\Omega)} \le & \mu h^{2p}  \|u^{n+1}\|_{H^{p+1}(\Omega)}^2.
\end{alignat*}
Therefore, combining the above inequalities we get
\begin{alignat}{1}
\tau \sum_{n=1}^{N-1} \|L^{n+1}\|_h^2 \le & C \big(\tau^4 \|\partial_t^3 u\|_{L^2([0,T], L^2(\Omega))}^2+ \|\beta\|_{\infty}^2  \tau^4 \|\partial_t^2  u\|_{L^2([0,T], H^1(\Omega))}^2\big) \nonumber \\
&+ C (\mu  h^{2p}+\gamma  \|\beta\|_{\infty}  h^{2(p+1/2)} +\|\nabla \beta \|_{\infty}^2  h^{2(p+1)} ) \tau \sum_{j=1}^{N-1} \|u^{n+1}\|_{H^{p+1}(\Omega)}^2. \label{L103}
\end{alignat}
The result now follows if we apply \eqref{cor1} and \eqref{cor2}. 
\end{proof}
\subsection{Error estimate for the material derivative}
In this section we prove error estimates for the material derivative.  We start with a lemma that shows that the projection of the material derivative superconverges. 
\begin{lemma}\label{lemmaMaterial}
Let $u$ be the solution of \eqref{weaku} and $\{u_h^n\}_{n=0}^N$ be the solution of
\eqref{BDF2weak}. Let $T=\tau N$, $p=1$, $\gamma >0$ and $\Pe >1$. Furthermore, suppose that $\Co$ is sufficiently small. If $m_h^{n+1}=D_{\tau} e_h^{n+1}+\beta \cdot \nabla \tilde{e}_h^{n+1}$ then
\begin{alignat}{1}
\tau \sum_{n=1}^{N-1} h \|\pi_h m_h^{n+1}\|^2 \le &  C  h((\gamma+1)  \|\beta \|_{\infty}+1)  \big(1+ T (\tau \|\nabla \beta\|_{\infty}^2+1)\big) \nonumber  \\
& \quad \times \big(h^{2(p+1)} (\|u(t_0)\|_{H^{p+1}(\Omega)}^2+\|u(t_0)\|_{H^{p+1}(\Omega)}^2) +\mathsf{G}\big)M \label{cormaterial1}  
\end{alignat}
\begin{proof}
Using \eqref{eherroreq} we get
\begin{alignat*}{1}
h \|\pi_h m_h^{n+1}\|^2=& h(m_h^{n+1}, \pi_h m_h^{n+1})_{\Omega} = S_1+S_2+S_3,
\end{alignat*}
where
\begin{alignat*}{1}
S_1=& -h\, a(e_h^{n+1},  \pi_h m_h^{n+1}),\\
S_2=&-h\,\gamma s(e_h^{n+1},  \pi_h m_h^{n+1}), \\
S_3=& h L^{n+1}( \pi_h m_h^{n+1}).
\end{alignat*}
We can easily show the following estimates
\begin{alignat*}{1}
S_1 \le&\sqrt{\mu} \|\sqrt{\mu} \nabla e_h^{n+1}\| \| \pi_h m_h^{n+1}\|, \\
S_2 \le &    C \sqrt{h} \sqrt{\gamma } \sqrt{\|\beta \|_{\infty}} \sqrt{\gamma} |e_h^{n+1}|_s \| \pi_h m_h^{n+1}\|,\\
S_3 \le & \|L^{n+1}\|_h  (\sqrt{\mu} +\sqrt{h}\sqrt{\gamma } \sqrt{\|\beta \|_{\infty}} + h)\| \pi_h m_h^{n+1}\|.
\end{alignat*}
Thus, we get 
\begin{alignat*}{1}
h \|\pi_h m_h^{n+1}\|^2\le &   C (\mu+ h\gamma \|\beta \|_{\infty}) (E(e_h^{n+1})^2 + \|L^{n+1}\|_h^2)  \\
& \le C h((\gamma+1)  \|\beta \|_{\infty}+1) (E(e_h^{n+1})^2 + \|L^{n+1}\|_h^2) 
\end{alignat*}
where we used that $\Pe >1$. 
This proves the following:
\begin{equation}
\tau \sum_{n=1}^{N-1} h \|\pi_h m_h^{n+1}\|^2 \le  C h((\gamma+1)  \|\beta \|_{\infty}+1)  (|||e_h|||^2+\tau \sum_{n=1}^{N-1}   \|L^{n+1}\|_h^2). \label{cormaterial1}  
\end{equation}
We have shown (see  \eqref{L103}):
\begin{equation*}
\tau \sum_{n=1}^{N-1} \|L^{n+1}\|_h^2 \le C \mathsf{G}. 
\end{equation*}

By using \eqref{corformat} we have 
\begin{alignat*}{1}
||||e_h|||^2  \le &     C \big(1+ T (\tau \|\nabla \beta\|_{\infty}^2+1)\big) \big(\|e_h^{0}\|^2 +\|e_h^{1}\|^2+ \tau \sum_{n=1}^{N-1} \|L^{n+1}\|_h^2\big)M.
\end{alignat*}
Thus, 
\begin{alignat}{1}\label{ehtriple}
||||e_h|||^2  \le &     C \big(1+ T (\tau \|\nabla \beta\|_{\infty}^2+1)\big) \big(h^{2(p+1)} (\|u(t_0)\|_{H^{p+1}(\Omega)}^2+\|u(t_0)\|_{H^{p+1}(\Omega)}^2) +\mathsf{G}\big)M.
\end{alignat}

\end{proof}

\end{lemma}

We can now prove an optimal estimate for the material derivative. 
\begin{theorem}
With the same hypothesis as in Lemma \ref{lemmaMaterial} the following estimate holds
\begin{alignat*}{1}
\tau \sum_{n=1}^{N-1} h \|m_h^{n+1}\|^2 \le &   C\Big(T h\|\nabla \beta\|_{\infty}^2  + \big(\|\beta\|_{\infty}+ h((\gamma+1)  \|\beta \|_{\infty}+1)\big) \big(1+T (\tau \|\nabla \beta\|_{\infty}^2+1)\big) \Big) \\
& \times  \Big(h^{p+1} (\|u(t_0)\|_{H^{p+1}(\Omega)}^2 +\|u(t_1)\|_{H^{p+1}(\Omega)}^2)+ \mathsf{G} \Big) \mathsf{M},
\end{alignat*}
where $\mathsf{G}$ and $\mathsf{M}$ are given in Theorem \ref{thmerrorestimate}. 
\end{theorem}

\begin{proof}
The triangle inequality gives
\begin{equation*}
\tau \sum_{n=1}^{N-1} h \|m_h^{n+1}\|^2 \le 2 \tau  \sum_{n=1}^{N-1} h \|\pi_h m_h^{n+1}-m_h^{n+1}\|^2+2 \tau \sum_{n=1}^{N-1} h \|\pi_h m_h^{n+1}\|^2.
\end{equation*}
Using \eqref{interpinq2} we have 
\begin{alignat*}{1}
\tau \sum_{n=1}^{N-1} h \|\pi_h m_h^{n+1}-m_h^{n+1}\|^2 \leq &
C\,(\|\nabla \beta\|_{\infty}^2 h \tau \sum_{n=1}^{N-1}\|e_h^{n+1}\|^2+  \tau  \|\beta\|_{\infty} \sum_{n=1}^{N-1} |e_h^{n+1}|_s^2) \\
\le & C \|\nabla \beta\|_{\infty}^2 h \tau \sum_{n=1}^{N-1}\|e_h^{n+1}\|^2+ C \|\beta\|_{\infty}  |||e_h|||^2.
\end{alignat*}

Applying \eqref{mainerror} we get
\begin{alignat*}{1}
\tau \sum_{n=1}^{N-1}\|e_h^{n+1}\|^2 \le  &      CT(h^{2(p+1)} (\|u(t_0)\|_{H^2(\Omega)}^2+\|u(t_1)\|_{H^2(\Omega)}^2)+ \mathsf{G}) \mathsf{M}
\end{alignat*}
Therefore, also using \eqref{ehtriple} we obtain
\begin{alignat*}{1}
&\tau \sum_{n=1}^{N-1} h \|\pi_h m_h^{n+1}-m_h^{n+1}\|^2    \\
&\le   C\Big(T h\|\nabla \beta\|_{\infty}^2  + \|\beta\|_{\infty}\big(1+T (\tau \|\nabla \beta\|_{\infty}^2+1)\big) \Big)\Big((h^{2(p+1)}(\|u(t_0)\|_{H^2(\Omega)}^2+\|u(t_1)\|_{H^2(\Omega)}^2)+ \mathsf{G}\Big) \mathsf{M}
\end{alignat*}
Thus, combining these inequalities with \eqref{cormaterial1} gives the result. 
\end{proof}
\section{Crank Nicolson IMEX scheme}\label{sec:crank}
In this section we will define the Crank Nicolson-IMEX method and prove that it is stable. 

The Crank Nicolson-IMEX method will read:
Find $u_h^{n+1} \in \mVh$ such that for $n \ge 1$,
\begin{equation}\label{CNscheme}
( \delta u_h^{n+1},v_h)_\Omega +\tau c_h(\hat{u}_h^{n+1}, v_h)+\tau a(\overline{u}_h^{n+1}, v_h) 
= \tau L^{n+1}(v_h),\quad \forall v_h \in \mVh,
\end{equation}
where $u_h^0$  and $u_h^1$ are given.  Here we use the notation
\begin{alignat*}{1}
\hat{u}_h^{n+1}:= & \frac{3}{2} u_h^{n}-\frac{1}{2} u_h^{n-1} =\frac{\tilde{u}_h^{n+1}+ u_h^{n}}{2}. \\
\overline{u}_h^{n+1}:=& \frac{u_h^{n+1}+ u_h^{n}}{2}.
\end{alignat*}

\flushleft
We see that
\begin{equation}\label{814}
\hat{u}_h^{n+1}=\overline{u}_h^{n+1}-\frac{1}{2}\delta \delta u_h^{n+1}.
\end{equation}
\flushleft
Here we define the triple norm as:
\[
|||v|||^2 := \sum_{n=1}^{N-1} \tau E(v^{n+1})^2 .
\]
\flushleft
In order to prove a stability result for the Crank-Nicoloson IMEX method we need to we need two different bounds for $\sum_{n=1}^{N-1}\|\delta \delta u_h^{n+1}\|^2$. One for the $\Pe >1$ and one for $\Pe \le 1$.  The first is as follows. 
\begin{lemma}\label{lemma818}
 Let $u_h$ solve \eqref{CNscheme}. If $\Pe>1$ and $\Co$ is sufficiently small,  then the following estimate holds
\begin{alignat}{1}
\sum_{n=1}^{N-1}\|\delta \delta u_h^{n+1}\|^2 \le   C \Co \sum_{n=1}^{N-1} \| \delta u_h^{n+1}-P_0(\delta u_h^{n+1})\|^2+C \tau  \sum_{n=1}^{N-1}  \|L^{n+1}\|_h^2+  C \Co ||| \overline{u}_h^{n+1}|||^2. \label{lemma818-1}
\end{alignat}
Moreover,  if $\Pe>1$  and $\max\{\Co_{4/3}, \Co\} $ is sufficiently small then we have the following estimate
\begin{alignat}{1}
\sum_{n=1}^{N-1}\|\delta \delta u_h^{n+1}\|^2 \le   C \tau  \sum_{n=1}^{N-1}  \|L^{n+1}\|_h^2+  C\Co ||| \overline{u}_h^{n+1}|||^2+\tau \sum_{n=1}^{N-1} \|u_h^{n+1}\|^2. \label{lemma818-2}
\end{alignat}
\end{lemma}
\begin{proof}
We see from \eqref{CNscheme} that
\begin{equation*}
( \delta \delta u_h^{n+1},v_h)_\Omega +\tau c_h(\delta \hat{u}_h^{n+1}, v_h)+\tau a(\delta \overline{u}_h^{n+1}, v_h) 
= \tau \delta L^{n+1}(v_h),\quad \forall v_h \in \mVh,
\end{equation*}
Therefore, we easily have 
\begin{alignat*}{1}
\sum_{n=1}^{N-1}\|\delta \delta u_h^{n+1}\|^2= S_1+S_2+S_3+S_4,
\end{alignat*}
where 
\begin{subequations}\label{S14}
\begin{alignat}{1}
S_1=&-\tau \sum_{n=1}^{N-1}  c(\delta \hat{u}_h^{n+1}, \delta \delta u_h^{n+1})  \\
S_2 =& -\tau \sum_{n=1}^{N-1}   \gamma s(\delta \hat{u}_h^{n+1}, \delta \delta u_h^{n+1}) \\
S_3 = &-\tau \sum_{n=1}^{N-1}  a(\delta \overline{u}_h^{n+1}, \delta \delta u_h^{n+1}) \\
S_4 = & \tau \sum_{n=1}^{N-1}  \delta  L^{n+1}( \delta \delta u_h^{n+1}).
\end{alignat}
\end{subequations}
We start with an estimate of $S_2$.  Using \eqref{814} and inverse estimates followed by Young's inequality we get:
\begin{alignat*}{1}
S_2 =& -\tau \sum_{n=1}^{N-1}   \gamma s(\delta \overline{u}_h^{n+1}, \delta \delta u_h^{n+1})+  \frac{\tau}{2} \sum_{n=1}^{N-1}   \gamma s(\delta \delta \delta u_h^{n+1}, \delta \delta u_h^{n+1}) \\
\le &  (\gamma C \Co+ \frac{1}{32}) \sum_{n=1}^{N-1}\|\delta \delta u_h^{n+1}\|^2+ \Co \tau  \sum_{n=1}^{N-1} \gamma |\overline{u}_h^{n+1}|_s^2.
\end{alignat*}
Similarly, we can show that
\begin{alignat*}{1}
S_3 \le    \frac{1}{32}  \sum_{n=1}^{N-1}\|\delta \delta u_h^{n+1}\|^2+  C \frac{\Co}{\Pe} \tau  \sum_{n=1}^{N-1} \|\sqrt{\mu} \nabla(\overline{u}_h^{n+1})\|^2.
\end{alignat*}
Hence, by our assumption that $\Pe >1$ we get that 
\begin{alignat*}{1}
S_2+S_3 \le  (C \sqrt{\Co}+\frac{1}{16})  \sum_{n=1}^{N-1}\|\delta \delta u_h^{n+1}\|^2+  C\Co ||| \overline{u}_h^{n+1}|||^2. 
\end{alignat*}

We can easily obtain using \eqref{preEbound1} and \eqref{preEbound2}
\begin{alignat*}{1}
S_4 \le    \frac{1}{8} \sum_{n=1}^{N-1}\|\delta \delta u_h^{n+1}\|^2+ C \tau  ( \frac{\sqrt{\Co}}{\sqrt{\Pe}}+ \sqrt{\gamma \Co} +\sqrt{\tau})^2 \sum_{n=1}^{N-1}  \|L^{n+1}\|_h^2.
\end{alignat*}

We now prove the estimate  \eqref{lemma818-1}. We easily can show that using an inverse estimate that
\begin{alignat*}{1}
S_1 \le  \frac{1}{8} \sum_{n=1}^{N-1}\|\delta \delta u_h^{n+1}\|^2+  C \Co \sum_{n=1}^{N-1} \| \delta u_h^{n+1}-P_0(\delta u_h^{n+1})\|^2.
\end{alignat*}

Thus, using that $\Pe >1$ and the fact that $\Co$ is sufficiently small we obtain:
\begin{alignat*}{1}
\sum_{n=1}^{N-1}\|\delta \delta u_h^{n+1}\|^2 \le   C \Co \sum_{n=1}^{N-1} \| \delta u_h^{n+1}-P_0(\delta u_h^{n+1})\|^2+C \tau  \sum_{n=1}^{N-1}  \|L^{n+1}\|_h^2+  C \sqrt{\Co} ||| \overline{u}_h^{n+1}|||^2. 
\end{alignat*}
This proves \eqref{lemma818-1}. 

To prove  \eqref{lemma818-2} we write $S_1$.
\begin{alignat*}{1}
S_1=&\tau \sum_{n=1}^{N-1}  c( \delta \delta u_h^{n+1},\delta \hat{u}_h^{n+1})=\tau \sum_{n=1}^{N-1}  ( \psi_h^{n+1},\delta \hat{u}_h^{n+1})_{\Omega}=\frac{3\tau}{2} \sum_{n=1}^{N-1}  ( \psi_h^{n+1},\delta u_h^{n})_{\Omega}- \frac{\tau}{2} \sum_{n=1}^{N-1}  ( \psi_h^{n+1},\delta u_h^{n-1})_{\Omega},
\end{alignat*}
where $\psi_h^n= C_h (\delta \delta u_h^{n})$.

We estimate  $( \psi_h^{n+1},\delta u_h^{n})_{\Omega}$.  Using \eqref{CNscheme} we write
\begin{equation*}
\tau  ( \psi_h^{n+1},\delta u_h^{n})_{\Omega}=M_1+M_2+M_3+M_4.
\end{equation*}
where 
\begin{alignat*}{1}
M_1:=& -\tau^2 c(\hat{u}_h^{n}, \psi_h^{n+1}), \\
M_2:= &- \tau^2\gamma s(\hat{u}_h^{n}, \psi_h^{n+1}), \\
M_3:=&-\tau^2 a(\overline{u}_h^{n}, \psi_h^{n+1}) ,\\
M_4:=&  \tau^2 L^{n}(\psi_h^{n+1}).
\end{alignat*}
Using inverse estimates and \eqref{eq:43cont} we obtain
\begin{alignat*}{1}
M_1 \le &     C \sqrt{\tau} (\Co_{4/3})^{3/2} \|\hat{u}_h^{n}\| \|\delta \delta u_h^{n+1}\| \\
\le& C (\Co_{4/3})^{3/2}   \tau \|\hat{u}_h^{n}\|^2+  C (\Co_{4/3})^{3/2}  \|\delta \delta u_h^{n+1}\|^2.
\end{alignat*}

To estimate $M_2$, we use inverse estimates and \eqref{ChCo} to obtain
\begin{alignat*}{1}
M_2= &- \tau^2\gamma s(\overline{u}_h^{n}, \psi_h^{n+1})+  \frac{\tau^2}{2} \gamma s(\delta \delta u_h^{n}, \psi_h^{n+1})\\
\le & C \gamma \Co^2 (\|\delta \delta u_h^{n}\|^2 +\|\delta \delta u_h^{n+1}\|^2)+ \tau \Co^2 \gamma |\overline{u}_h^{n}|_s^2. 
\end{alignat*}
Similarly, we get 
\begin{alignat*}{1}
M_3 \le & C  \frac{\Co}{\sqrt{\Pe}}  \|\delta \delta u_h^{n+1}\|^2+ \tau \frac{\Co}{\sqrt{\Pe}} \|\sqrt{\mu} \nabla (\overline{u}_h^{n})\|^2. 
\end{alignat*}
Finally, again using   inverse estimates and \eqref{ChCo} to obtain
\begin{alignat*}{1}
M_4 \le & C  (\frac{\Co^3}{\Pe}+ \gamma^2 \Co^3+\Co^2 )  \|\delta \delta u_h^{n+1}\|^2+ \frac{\tau}{2} \|L^n\|_h^2.
\end{alignat*}
Hence, using that $\Pe >1$ and that we can take $\Co \le 1$ gives 
\begin{alignat*}{1}
\frac{3\tau}{2} \sum_{n=1}^{N-1}  ( \psi_h^{n+1},\delta u_h^{n})_{\Omega} \le &   C (\Co_{4/3})^{3/2}+ \Co+ \Co^3 \gamma^2) \sum_{n=1}^{N-1}    \|\delta \delta u_h^{n+1}\|^2+ \Co^2 ||| \overline{u}_h|||^2  \\
&   + \frac{\tau}{2} \sum_{n=1}^{N-1}\|L^{n+1}\|_h^2    +  C (\Co_{4/3})^{3/2}   \tau  \sum_{n=1}^{N-1}\|u^{n+1}\|^2.
\end{alignat*}
Similarly, we can prove the same estimate for $- \frac{\tau}{2} \sum_{n=1}^{N-1}  ( \psi_h^{n+1},\delta u_h^{n-1})_{\Omega}$ and so we get
\begin{alignat*}{1}
S_1 \le &  C \big (\Co_{4/3})^{3/2}+ \Co+ \Co^3 \gamma^2\big) \sum_{n=1}^{N-1}    \|\delta \delta u_h^{n+1}\|^2+ \Co^2 ||| \overline{u}_h|||^2  \\
&   + \tau \sum_{n=1}^{N-1}\|L^{n+1}\|_h^2 +  C (\Co_{4/3})^{3/2}   \tau  \sum_{n=1}^{N-1} \|u^{n+1}\|^2.
\end{alignat*}
If we combine the estimates of $S_i's$ and take $\max\{\Co_{4/3}, \Co\}$ we obtain \eqref{lemma818-2}. 

\end{proof}

The following alternative estimate will be useful when $\Pe \le 1$. 
\begin{lemma}\label{lemma819}
Assume  that $\tau \le 1$ and $\Co$ is sufficiently small.   Let $u_h$ solve \eqref{CNscheme} then the following estimate holds
\begin{alignat*}{1}
\sum_{n=1}^{N-1}\|\delta \delta u_h^{n+1}\|^2 \le&  -\tau (\|\sqrt{\mu} \nabla(\delta u_h^{N})\|^2-\|\sqrt{\mu} \nabla(\delta u_h^{1})\|^2) + 2  \sum_{n=0}^{N-1}  \|\delta L^{n+1}\|_h^2  \\
 &+C \tau \Co  (\Pe+1)  \sum_{n=1}^{N-1} E( \overline{u}_h^{n+1})^2+  \tau^2 \sum_{n=1}^{N-1} \|\sqrt{\mu} \nabla (\delta u_h^{n+1})\|^2.
\end{alignat*}
\end{lemma}
\begin{proof}
\begin{alignat*}{1}
\sum_{n=1}^{N-1}\|\delta \delta u_h^{n+1}\|^2= S_1+S_2+S_3+S_4,
\end{alignat*}
where  $S_i's$ are given in \eqref{S14}. 

We first notice that   $a(\delta \overline{u}_h^{n+1}, \delta \delta u_h^{n+1})= \frac{1}{2} \big(a(\delta u_h^{n+1}, \delta u_h^{n+1})-a(\delta u_h^{n}, \delta u_h^{n})\big)$. Hence, 
\begin{equation*}
S_3= -\frac{\tau}{2} (\|\sqrt{\mu} \nabla(\delta u_h^{N})\|^2-\|\sqrt{\mu} \nabla(\delta u_h^{1})\|^2).
\end{equation*}
Using \eqref{814} we obtain
\begin{alignat*}{1}
S_1 = -\tau \sum_{n=1}^{N-1}  c(\delta \overline{u}_h^{n+1}, \delta \delta u_h^{n+1})  +\frac{\tau}{2} \sum_{n=1}^{N-1}  c(\delta (\delta \delta u_h^{n+1}), \delta \delta u_h^{n+1}).
\end{alignat*}
We then see that
\begin{alignat*}{1}
-\tau \sum_{n=1}^{N-1}  c(\delta \overline{u}_h^{n+1}, \delta \delta u_h^{n+1}) \le&  \frac{C  \tau \|\beta\|_{\infty}}{\sqrt{\mu}}  \sum_{n=1}^{N-1} \|\sqrt{\mu} \nabla(\delta \overline{u}_h^{n+1})\| \| \delta \delta u_h^{n+1}\| \\
\le & C \tau \Co \Pe \sum_{n=1}^{N-1} \|\sqrt{\mu} \nabla(\delta \overline{u}_h^{n+1})\|^2+ \frac{1}{16}\sum_{n=1}^{N-1}\|\delta \delta u_h^{n+1}\|^2.
\end{alignat*}
Similarly, now using inverse estimates, we get 
\begin{alignat*}{1}
\frac{\tau}{2} \sum_{n=1}^{N-1}  c(\delta (\delta \delta u_h^{n+1}), \delta \delta u_h^{n+1}) \le &  C \Co \sum_{n=1}^{N-1}\|\delta (\delta \delta u_h^{n+1})\| \|\delta \delta u_h^{n+1}\| \\
\le &  C \Co \sum_{n=1}^{N-1}\|\delta \delta u_h^{n+1}\|^2.
\end{alignat*}

Hence, we have shown that 
\begin{alignat*}{1}
S_1 \le   C \tau \Co \Pe \sum_{n=1}^{N-1} \|\sqrt{\mu} \nabla(\delta \overline{u}_h^{n+1})\|^2+ (\frac{1}{16}+ C \Co) \sum_{n=1}^{N-1}\|\delta \delta u_h^{n+1}\|^2.
\end{alignat*}
In a very similar fashion we can  prove that 
\begin{alignat*}{1}
S_2 \le   C \tau \Co  \sum_{n=1}^{N-1} \gamma | \delta \overline{u}_h^{n+1}|_s^2+ (\frac{1}{16}+ C \gamma \Co) \sum_{n=1}^{N-1}\|\delta \delta u_h^{n+1}\|^2.
\end{alignat*}
Therefore, 

\begin{alignat*}{1}
S_1+S_2 \le   (\frac{1}{8}+ C (\gamma+1) \Co) \sum_{n=1}^{N-1} \|\delta \delta u_h^{n+1}\|^2+ C \tau \Co  (\Pe+1)  \sum_{n=1}^{N-1} E( \overline{u}_h^{n+1})^2. 
\end{alignat*}
Here we used that  $\sum_{n=1}^{N-1} E( \delta \overline{u}_h^{n+1})^2  \le C \sum_{n=1}^{N-1} E( \overline{u}_h^{n+1})^2$. 

To bound $S_4$ we use the definition of the operator norm.
\begin{alignat*}{1}
S_4 \le  &\tau \sum_{n=1}^{N-1} \sqrt{E(\delta \delta u_h^{n+1})^2+ \|\delta \delta u_h^{n+1}\|^2} \, \|\delta L^{n+1}\|_h  \\
    \le & \sum_{n=1}^{N-1}  \|\delta L^{n+1}\|_h^2+  \frac{\tau^2}{4} \sum_{n=1}^{N-1} E(\delta \delta u_h^{n+1})^2+ \frac{\tau^2}{4}  \sum_{n=1}^{N-1} \|\delta \delta u_h^{n+1}\|^2. 
\end{alignat*}
We can bound the energy as follows.
\begin{alignat*}{1}
 \frac{\tau^2}{4} \sum_{n=1}^{N-1} E(\delta \delta u_h^{n+1})^2 =  & \frac{\tau^2}{4} \sum_{n=1}^{N-1} \|\sqrt{\mu} \nabla (\delta \delta u_h^{n+1})\|^2+   \frac{\tau^2}{4} \sum_{n=1}^{N-1} \gamma | \delta \delta u_h^{n+1}|_s^2 \\
 \le &   \tau^2 \sum_{n=1}^{N-1} \|\sqrt{\mu} \nabla (\delta u_h^{n+1})\|^2+  C \Co \tau   \sum_{n=1}^{N-1}  \|\delta \delta u_h^{n+1}\|^2. 
 \end{alignat*}
 Hence, 
\begin{alignat*}{1}
S_4 \le  \sum_{n=1}^{N-1}  \|\delta L^{n+1}\|_h^2 + \tau^2 \sum_{n=1}^{N-1} \|\sqrt{\mu} \nabla (\delta u_h^{n+1})\|^2+ (\frac{\tau^2}{4}+ C \Co \tau) \sum_{n=1}^{N-1} \|\delta \delta u_h^{n+1}\|^2.
\end{alignat*}
 We arrive at 
 
 \begin{alignat*}{1}
\sum_{n=1}^{N-1}\|\delta \delta u_h^{n+1}\|^2 \le&  -\frac{\tau}{2} (\|\sqrt{\mu} \nabla(\delta u_h^{N})\|^2-\|\sqrt{\mu} \nabla(\delta u_h^{2})\|^2) +  \sum_{n=1}^{N-1}  \|\delta L^{n+1}\|_h^2  \\
 &+C \tau \Co  (\Pe+1)  \sum_{n=1}^{N-1} E( \overline{u}_h^{n+1})^2+  \tau^2 \sum_{n=1}^{N-1} \|\sqrt{\mu} \nabla (\delta u_h^{n+1})\|^2 \\ 
 &+(\frac{\tau^2}{4}+\frac{1}{8}+ C \Co(1+ \tau)) \sum_{n=1}^{N-1} \|\delta \delta u_h^{n+1}\|^2.
\end{alignat*}
The result follows by taking $\Co$ sufficiently small so that $(\frac{\tau^2}{4}+\frac{1}{8}+ C \Co(1+ \tau)) \le \frac{1}{2}$. 
\end{proof}

We will need an auxiliarly lemma in the case $\Pe>1$.
In the case $p=1$ (and $\Pe >1$) we will need an auxiliarly result. 
\begin{lemma}\label{lemma817}
Let $u_h$ solve \eqref{CNscheme}. 

If $p=1$ and $\gamma>0$ then the following estimate holds
\begin{alignat}{1}
\|\delta u_h^{n+1}-P_0(\delta u_h^{n+1})\| \le&  C \sqrt{ \tau\Co} (\frac{1}{\sqrt{\Pe}}+ \sqrt{\gamma}+ \frac{1}{\sqrt{\gamma}}) E(\overline{u}_h^{n+1}) +  C \gamma \Co \| \delta \delta u_h^{n+1}\|  \nonumber \\
&+  C \tau \|\nabla \beta\|_{\infty} \|\hat u_h^{n+1}\|+ C\sqrt{\tau} ( \frac{\sqrt{\Co}}{\sqrt{\Pe}}+ \sqrt{\gamma \Co} +\sqrt{\tau}) \|L^{n+1}\|_h.\label{lemma817-1}
\end{alignat}

If  $p \ge 1$
\begin{alignat}{1}
\|\delta u_h^{n+1}\| \le&  C \sqrt{ \tau\Co} (\frac{1}{\sqrt{\Pe}}+ \sqrt{\gamma}+ \frac{1}{\sqrt{\gamma}}) E(\overline{u}_h^{n+1}) +  C \gamma \Co \| \delta \delta u_h^{n+1}\| \nonumber \\
&+ C\sqrt{\tau} ( \frac{\sqrt{\Co}}{\sqrt{\Pe}}+ \sqrt{\gamma \Co} +\sqrt{\tau}) \|L^{n+1}\|_h + C \|\beta\|_{\infty} \frac{\tau}{h} \| \hat u_h^{n+1}\|. \label{lemma817-2}
\end{alignat}

\end{lemma}

\begin{proof}
We first prove \eqref{lemma817-1}.  Let $y_h=\delta u_h^{n+1}$ and then we have by \eqref{CNscheme}
\begin{alignat*}{1}
\|y_h-P_0 y_h\|^2=&( y_h, y_h-P_0(y_h))_{\Omega}\\
=&(y_h, \pi_h( y_h-P_0(y_h)))_{\Omega}\\
=& -\tau c_h(\hat{u}_h^{n+1}, \pi_h( y_h-P_0(y_h)))-\tau a( \overline{u}_h^{n+1}, \pi_h( y_h-P_0(y_h))) \\
&+\tau L^{n+1}(\pi_h( y_h-P_0(y_h))).
\end{alignat*}
We use the Cauchy-Schwarz inequality followed inverse estimates to bound the symmetric terms
\begin{alignat*}{1}
- \tau  a( \overline{u}_h^{n+1}, \pi_h( y_h-P_0(y_h)))\le& C \frac{\sqrt{\mu}\tau}{h} \|\sqrt{\mu}  \nabla \overline{u}_h^{n+1}\| \, \| y_h-P_0(y_h)\| \\
 \le& C \frac{\sqrt{\tau \Co}}{\sqrt{\Pe}} \|\sqrt{\mu}  \nabla \overline{u}_h^{n+1}\| \, \| y_h-P_0(y_h)\|,
\end{alignat*}
and for the stabilization we apply \eqref{eq:s_cont}
\begin{alignat}{1}
-\tau   \gamma s(\hat{u}_h^{n+1},\pi_h( y_h-P_0(y_h))) 
\le&   C\sqrt{\tau} \sqrt{\gamma \Co}         \sqrt{\gamma} | \hat{u}_h^{n+1}|_s \| y_h-P_0(y_h)\|  \nonumber \\
  \le & C(\sqrt{\tau} \sqrt{\gamma \Co}         \sqrt{\gamma} | \overline{u}_h^{n+1}|_s+  \gamma \Co \| \delta \delta u_h^{n+1}\|)  \| y_h-P_0(y_h)\|.   \label{eq:stab_contCN}
\end{alignat}
\flushleft
Next we bound $\tau L^{n+1}(\pi_h( y_h-P_0(y_h)))$ using \eqref{eq:Ldef} and \eqref{eq:Ebound}.
\begin{alignat*}{1}
\tau L^{n+1}(\pi_h( y_h-P_0(y_h))) \le &{\tau}\|L^{n+1}\|_h \sqrt{ E( \pi_h( y_h-P_0(y_h)))^2+ \|\pi_h( y_h-P_0(y_h))\|^2} \\
 \le &C\sqrt{\tau} ( \frac{\sqrt{\Co}}{\sqrt{\Pe}}+ \sqrt{\gamma \Co} +\sqrt{\tau}) \|L^{n+1}\|_h   \|y_h-P_0(y_h)\|.
\end{alignat*}
\flushleft
It only remains to bound the contribution from the form $c$.
\begin{alignat*}{1}
-\tau c(\hat u_h^{n+1}, \pi_h( y_h-P_0(y_h))) =&-\tau ((\beta-\beta_0) \cdot \nabla \hat u_h^{n+1}, \pi_h( y_h-P_0(y_h)))_{\Omega} \\
&-\tau (\beta_0 \cdot \nabla \hat u_h^{n+1}-w_h^{n+1}, (I-\pi_h)( y_h-P_0(y_h)))_{\Omega}.
\end{alignat*}
Here $w_h \in \mVh$ is arbitrary. Note that we crucially used that $p=1$ which implies that $\beta_0 \cdot \nabla \hat{u}_h^{n+1} \in W_h$.
\flushleft
Hence, using \eqref{beta0} and \eqref{interpinq} we obtain
\begin{alignat*}{1}
&-\tau c(\hat u_h^{n+1}, \pi_h( y_h-P_0(y_h)))\\
 \le & C \tau (\|\nabla \beta\|_{\infty}\|\hat u_h^{n+1}\|+ \|\beta\|_{\infty}^{1/2} h^{-1/2} | \hat u_h^{n+1}|_s)  \|y_h-P_0(y_h)\| \\
\le & C \sqrt{\tau} (\sqrt{\tau} \|\nabla \beta\|_{\infty} \|\hat u_h^{n+1}\|+ \sqrt{\Co} \frac{1}{\sqrt{\gamma}}  \sqrt{\gamma}| \hat u_h^{n+1}|_s)  \|y_h-P_0(y_h)\|.
\end{alignat*}
\flushleft
Finally, by the triangle inequality and \eqref{eq:s_cont} we have the bound $| \hat u_h^{n+1}|_s \le | \overline{u}_h^{n+1}|_s+ \frac{C \sqrt{\Co}}{\sqrt{\tau}} \|\delta \delta u_h^{n+1}\|$.  Combining the above inequalities gives \eqref{lemma817-1}.
\flushleft
Now we prove  \eqref{lemma817-2}. Using \eqref{CNscheme} we have
\begin{alignat*}{1}
\|y_h\|^2= -\tau c(\hat{u}_h^{n+1},  y_h)- \tau\gamma  s(\hat{u}_h^{n+1},  y_h) -\tau a( \overline{u}_h^{n+1}, y_h) +\tau L^{n+1}( y_h).
\end{alignat*}
Similar to what  we did above we can show that 
\begin{alignat*}{1}
& - \tau\gamma  s(\hat{u}_h^{n+1},  y_h) -\tau a( \overline{u}_h^{n+1}, y_h) +\tau L^{n+1}( y_h)  \\
\le &  C \frac{\sqrt{\tau \Co}}{\sqrt{\Pe}} \|\sqrt{\mu}  \nabla \overline{u}_h^{n+1}\| \, \| y_h\| \\
&+  C(\sqrt{\tau} \sqrt{\gamma \Co}         \sqrt{\gamma} | \overline{u}_h^{n+1}|_s+  \gamma \Co \| \delta \delta u_h^{n+1}\|)  \| y_h\|  \\
&+C\sqrt{\tau} ( \frac{\sqrt{\Co}}{\sqrt{\Pe}}+ \sqrt{\gamma \Co} +\sqrt{\tau}) \|L^{n+1}\|_h \| y_h\|.
\end{alignat*}
We then bound the remaining term using inverse estimates
 \begin{alignat*}{1}
 -\tau c(\hat{u}_h^{n+1},  y_h) \le \frac{C \tau}{h} \| \beta\|_{\infty}\|\hat{u}_h^{n+1} \| \| y_h\|.
\end{alignat*}
Combining the above estimates proves  \eqref{lemma817-2}.
\end{proof}

\begin{corollary}\label{cor615}
Let $u_h$ solve \eqref{CNscheme}. 

Let $p=1$,  $\Pe>1$ and $\gamma>0$. If  $\Co$ is sufficiently small the following estimate holds
\begin{alignat}{1}
\sum_{n=1}^{N-1}\|\delta \delta u_h^{n+1}\|^2 + \sum_{n=1}^{N-1} \|\delta u_h^{n+1}-P_0(\delta u_h^{n+1})\|^2 \le&  C \Co |||\overline{u}_h^{n+1}|||^2
 + C \tau^2 \|\nabla \beta\|_{\infty}^2  \sum_{n=0}^{N-1}\|u_h^{n+1}\|^2 \nonumber \\
&+ C\tau   \sum_{n=0}^{N-1} \|L^{n+1}\|_h^2. \label{cor615-1}
\end{alignat}
Let $p\ge 1$ and  $\Pe>1$. If  $\max\{\Co, \Co_{4/3}\}$ is sufficiently small the following estimate holds
\begin{alignat}{1}
 \sum_{n=1}^{N-1} \|\delta u_h^{n+1}\|^2 \le&  C \Co |||\overline{u}_h|||^2 + C\tau   \sum_{n=1}^{N-1} \|L^{n+1}\|_h^2  \nonumber \\
 & +  C(\sqrt{\tau}  (\Co_{4/3})^2+\tau)  \sum_{n=1}^{N-1} \|u_h^{n+1}\|^2.  \label{cor615-2}
\end{alignat}

\end{corollary} 
\begin{proof}
We first prove \eqref{cor615-1}. From  \eqref{lemma817-1} we get
\begin{alignat*}{1}
 \sum_{n=1}^{N-1} \|\delta u_h^{n+1}-P_0(\delta u_h^{n+1})\|^2 \le&  C \Co |||\overline{u}_h|||^2
 + C \tau^2 \|\nabla \beta\|_{\infty}^2  \sum_{n=1}^{N-1}\|u_h^{n+1}\|^2 \\
&+ C\tau   \sum_{n=1}^{N-1} \|L^{n+1}\|_h^2+C  \sum_{n=0}^{N-1}\|\delta \delta u_h^{n+1}\|^2. 
\end{alignat*}
Using \eqref{lemma818-1} we thus obtain
\begin{alignat*}{1}
 \sum_{n=1}^{N-1} \|\delta u_h^{n+1}-P_0(\delta u_h^{n+1})\|^2 \le&  C \Co |||\overline{u}_h^{n+1}|||^2
 + C \tau^2 \|\nabla \beta\|_{\infty}^2  \sum_{n=1}^{N-1}\|u_h^{n+1}\|^2 \\
&+ C\tau   \sum_{n=1}^{N-1} \|L^{n+1}\|_h^2+C\Co  \sum_{n=0}^{N-1} \|\delta u_h^{n+1}-P_0(\delta u_h^{n+1})\|^2.
\end{alignat*}
If $\Co$ is sufficiently small we get 
\begin{alignat*}{1}
 \sum_{n=1}^{N-1} \|\delta u_h^{n+1}-P_0(\delta u_h^{n+1})\|^2 \le&  C \Co |||\overline{u}_h^{n+1}|||^2
 + C \tau^2 \|\nabla \beta\|_{\infty}^2  \sum_{n=1}^{N-1}\|u_h^{n+1}\|^2 \\
&+ C\tau   \sum_{n=1}^{N-1} \|L^{n+1}\|_h^2.
\end{alignat*}
The bound for $\sum_{n=1}^{N-1}\|\delta \delta u_h^{n+1}\|^2$ follows from this and using  \eqref{lemma818-1} again.

Now we prove  \eqref{cor615-2}. From  \eqref{lemma817-2} we get
\begin{alignat*}{1}
 \sum_{n=1}^{N-1} \|\delta u_h^{n+1}\|^2 \le&  C \Co |||\overline{u}_h|||^2 + C\tau   \sum_{n=1}^{N-1} \|L^{n+1}\|_h^2+C  \sum_{n=0}^{N-1}\|\delta \delta u_h^{n+1}\|^2 \\
 & +  \frac{C \tau^2 \|\beta\|_{\infty}^2 }{h^2}  \sum_{n=1}^{N-1} \|u_h^{n+1}\|^2.
\end{alignat*}
The estimate \eqref{cor615-2} now follows from  \eqref{lemma818-2} and using the definiton of $\Co_{4/3}$.
\end{proof}

\begin{theorem}\label{thm:stabCN}
Let  $T=N \tau$. Suppose that $\Co$ is chosen sufficiently small only depending on geometric constants of the mesh and $\gamma$.  For $\{u_h^n\}$ solving \eqref{CNscheme} we have the following bounds:

If $\Pe \le 1$ then for all $p \ge 1$,
\begin{alignat}{1}
\|u_h^{N}\|^2+ \tau \|\sqrt{\mu} \nabla(\delta u_h^{N})\|^2 \le (1+  Te^{T}) M, \label{thmCN1}
\end{alignat}
where
\begin{alignat*}{1}
M=& \tau \|\sqrt{\mu} \nabla(\delta u_h^{1})\|^2+ \|u_h^{1}\|^2+16  \tau \sum_{n=1}^{N-1} \|L^{n+1}\|_h^2   +4   \sum_{n=1}^{N-1} \|\delta L^{n+1}\|_h^2 .
\end{alignat*}

If $\Pe > 1$  and $p=1$, $\gamma>0$ then 
\begin{alignat}{1}
\|u_h^{N}\|^2  \le &    (2\|u_h^{1}\|^2+ \|u_h^{2}\|^2+ 36 \tau \sum_{n=1}^{N-1} \|L^{n+1}\|_h^2) M,  \label{thmCN2}
\end{alignat}
where 
\begin{equation}\label{Mex}
M=\Big(1+T (\tau \|\nabla \beta\|_{\infty}^2+\frac{c_L}{8}) e^{ T (\tau \|\nabla \beta\|_{\infty}^2+\frac{c_L}{8})}\Big). 
\end{equation}

If $\Pe > 1$  and $p \ge 1$,  and $\max \{\Co, \Co_{4/3}\}$ is sufficiently small we have
\begin{alignat}{1}
\|u_h^{N}\|^2  \le &    (2\|u_h^{1}\|^2+ \|u_h^{2}\|^2+ 36 \tau \sum_{n=1}^{N-1} \|L^{n+1}\|_h^2) M,  \label{thmCN3}
\end{alignat}
with $M$ as in \eqref{Mex}.

\end{theorem}
\begin{proof}
Choose $v_h=\overline{u}_h^{n+1}$ in \eqref{CNscheme} and use \eqref{814}, \eqref{cgamma} to get
\end{proof}
\begin{alignat}{1}
\frac{1}{2}\|u_h^{n+1}\|^2-\frac{1}{2} \|u_h^{n}\|^2+ \tau E(\overline{u}_h^{n+1})^2 = \frac{\tau}{2}  c_h(\delta \delta u_h^{n+1}, \overline{u}_h^{n+1})+  \tau L^{n+1}(\overline{u}_h^{n+1}), \label{619}
\end{alignat}
Taking the sum from $1 \le n \le N-1$ in \eqref{619} we arrive at
\begin{alignat}{1}
&
  \frac{1}{2}\|u_h^{N}\|^2 - \frac{1}{2}\|u_h^{1}\|^2 + ||| \overline{u}_h|||^2 =  S_1+S_2 +S_3, \label{614}
\end{alignat}
where
\begin{alignat*}{1} 
S_1:=& \frac{\tau}{2} \sum_{n=1}^{N-1}  c(\delta \delta u_h^{n+1}, \overline{u}_h^{n+1})  \\
S_2:=& \frac{\tau}{2} \sum_{n=1}^{N-1}  \gamma s(\delta \delta u_h^{n+1}, \overline{u}_h^{n+1})  \\
S_3:=& \tau \sum_{n=1}^{N-1} L^{n+1}(\overline{u}_h^{n+1}).\\
\end{alignat*}

Let us estimate $S_3$. We have 
\begin{alignat*}{1}
S_3 \le & \tau \sum_{n=1}^{N-1} \|L^{n+1}\|_h \sqrt{E(\overline{u}_h^{n+1})^2+\|\overline{u}_h^{n+1}\|^2} \\
\le &  8 \tau \sum_{n=1}^{N-1} \|L^{n+1}\|_h^2 + \frac{1}{32} ||| \overline{u}_h|||^2+ \frac{c_L}{32} \tau  \sum_{n=1}^{N-1} \|u_h^{n+1}\|^2. 
\end{alignat*}

We now consider three cases. 

{\bf Case 1: $\Pe < 1$, $p \ge 1$}. 

In this case, we can easily show that 
\begin{alignat*}{1}
S_2=&  -\frac{\tau}{2} \sum_{n=1}^{N-1}  c( \overline{u}_h^{n+1}, \delta \delta u_h^{n+1})  \\
 \le &  \frac{\|\beta\|_{\infty} \tau}{\sqrt{\mu}}   \sum_{n=1}^{N-1} \| \sqrt{\mu} \nabla  (\overline{u}_h^{n+1})\|   \, \|  \delta \delta u_h^{n+1}\| \\
\le & \sqrt{\Co \Pe} \sqrt{\tau}   \sum_{n=1}^{N-1} \| \sqrt{\mu} \nabla  (\overline{u}_h^{n+1})\|   \, \|  \delta \delta u_h^{n+1}\| \\
\le &  \tau  C \Co \sum_{n=1}^{N-1} \| \sqrt{\mu} \nabla  (\overline{u}_h^{n+1})\|^2  +  \frac{1}{4}\sum_{n=1}^{N-1}  \|  \delta \delta u_h^{n+1}\|^2.
\end{alignat*}

Similarly,  
\begin{alignat*}{1}
S_1=&  C \sqrt{\Co } \sqrt{\tau}   \sum_{n=1}^{N-1} \gamma |\overline{u}_h^{n+1}|_s \| \|  \delta \delta u_h^{n+1}\| \\
\le &  C \Co \gamma \tau \sum_{n=1}^{N-1} \gamma |\overline{u}_h^{n+1}|_s^2 +  \frac{1}{4} \sum_{n=1}^{N-1}  \|  \delta \delta u_h^{n+1}\|^2.
\end{alignat*}

Using \eqref{614} and combining the above inequalities we get
\begin{alignat*}{1}
\frac{1}{2}\|u_h^{N}\|^2 - \frac{1}{2}\|u_h^{1}\|^2 + ||| \overline{u}_h|||^2 \le &  \left(\frac{1}{32}+ \Co(1 +\gamma)\right)  ||| \overline{u}_h|||^2 + 8 \tau \sum_{n=1}^{N-1} \|L^{n+1}\|_h^2  \\
 &+\frac{c_L}{32} \tau  \sum_{n=1}^{N-1} \|u_h^{n+1}\|^2 + \frac{1}{2}\sum_{n=1}^{N-1}  \|  \delta \delta u_h^{n+1}\|^2.
\end{alignat*}
If we use Lemma \ref{lemma819} and the fact that $\Pe \le 1 $ we get 

\begin{alignat*}{1}
\frac{1}{2}\|u_h^{N}\|^2+ \frac{\tau}{2} \|\sqrt{\mu} \nabla(\delta u_h^{N})\|^2  + ||| \overline{u}_h|||^2 \le & \frac{\tau}{2} \|\sqrt{\mu} \nabla(\delta u_h^{1})\|^2+ \frac{1}{2}\|u_h^{1}\|^2+8  \tau \sum_{n=0}^{N-1} \|L^{n+1}\|_h^2 \\
 &  +2   \sum_{n=1}^{N-1} \|\delta L^{n+1}\|_h^2  +  \tau^2 \sum_{n=0}^{N-1} \|\sqrt{\mu} \nabla (\delta u_h^{n+1})\|^2  \\
 & + \frac{c_L}{32} \tau  \sum_{n=1}^{N-1} \|u_h^{n+1}\|^2 +  \left(\frac{1}{32}+C  \Co(1 +\gamma)\right)  ||| \overline{u}_h|||^2. 
\end{alignat*}
Choosing $\Co$ sufficiently small gives 
\begin{alignat*}{1}
\frac{1}{2}\|u_h^{N}\|^2+ \frac{\tau}{2} \|\sqrt{\mu} \nabla(\delta u_h^{N})\|^2  + \frac{1}{2}||| \overline{u}_h|||^2 \le & \frac{\tau}{2} \|\sqrt{\mu} \nabla(\delta u_h^{2})\|^2+ \frac{1}{2}\|u_h^{1}\|^2+8  \tau \sum_{n=0}^{N-1} \|L^{n+1}\|_h^2 \\
 &  +2   \sum_{n=1}^{N-1} \|\delta L^{n+1}\|_h^2  +  \tau^2 \sum_{n=1}^{N-1} \|\sqrt{\mu} \nabla (\delta u_h^{n+1})\|^2  \\
 & + \frac{c_L}{32} \tau  \sum_{n=1}^{N-1} \|u_h^{n+1}\|^2.
\end{alignat*}
Applying Gronwall's inequality gives \eqref{thmCN1}.

{\bf Case 2: $\Pe  > 1$, $p = 1$, $\gamma>0$}.
\flushleft
Now we estimate $S_2$.  Using the arithmetic-geometric mean inequality and inverse estimates we obtain
\begin{alignat*}{1}
S_2 \le &  C \tau  \sum_{n=1}^{N-1}    \gamma \frac{\sqrt{\|\beta\|_{\infty}}}{\sqrt{h}}  |\overline{u}_h^{n+1}|_s \|\delta \delta u_h^{n+1}\| \\
\le &  \frac{1}{32} ||| \overline{u}_h|||^2+ C \gamma \Co  \sum_{n=1}^{N-1} \|\delta \delta u_h^{n+1}\|^2.
\end{alignat*}

\flushleft
To estimate $S_1$ we  re-write it as follows
\begin{alignat*}{1} 
S_1=D_1+D_2,
\end{alignat*}
where 
\begin{alignat*}{1}
D_1:=& \frac{\tau}{2} \sum_{n=1}^{N-1}  c(\delta \delta u_h^{n+1}, u_h^{n+1}), \\
D_2:=&-\frac{\tau}{2} \sum_{n=1}^{N-1}  c(\delta \delta u_h^{n+1}, \delta u_h^{n+1}).
\end{alignat*}
We use integration by parts and inverse estimates to obtain
\begin{alignat*}{1}
D_2=&\frac{\tau}{2} \sum_{n=1}^{N-1}  c(\delta u_h^{n+1}, \delta \delta u_h^{n+1}) \\
\le  & C \Co  \sum_{n=1}^{N-1}  \|  \delta u_h^{n+1}-P_0( \delta u_h^{n+1})\|  \, \| \delta \delta u_h^{n+1}\| \\
\le & C \Co  \sum_{n=1}^{N-1}  \|  \delta u_h^{n+1}-P_0( \delta u_h^{n+1})\|^2 +  C \Co \sum_{n=1}^{N-1}  \| \delta \delta u_h^{n+1}\|^2.
\end{alignat*}

\flushleft
To estimate $D_1$ we use summation by parts \eqref{rbi1} to write
\begin{alignat*}{1} 
D_1=& -\frac{\tau}{2} \sum_{n=2}^{N-1}  c(\delta u_h^{n}, \delta u_h^{n+1}) +\frac{\tau}{2} \big(c(\delta u_h^N,  u_h^{N})- c(\delta u_h^1,  u_h^{2})\big) \\
 =&-\frac{\tau}{2} \sum_{n=2}^{N-1}  c(\delta u_h^{n}, \delta (u_h^{n+1}-u_h^n)) +\frac{\tau}{2} \big(c(\delta u_h^N,  u_h^{N})- c(\delta u_h^1, u_h^{2})\big)  \\
=&  -\frac{\tau}{2} \sum_{n=2}^{N-1}  c(\delta u_h^{n}, \delta \delta u_h^{n+1}) +\frac{\tau}{2} \big(c(\delta u_h^N,  u_h^{N})- c(\delta u_h^1,  u_h^{2})\big) . 
\end{alignat*}
We easily have
\begin{equation*}
-\frac{\tau}{2} \sum_{n=2}^{N-1}  c(\delta u_h^{n}, \delta \delta u_h^{n+1}) \le  C \Co  \sum_{n=1}^{N-1}  \|  \delta u_h^{n+1}-P_0( \delta u_h^{n+1})\|^2 +  C \Co \sum_{n=1}^{N-1}  \| \delta \delta u_h^{n+1}\|^2.
\end{equation*}
The next term can similarly be bounded:
\begin{alignat*}{1}
\frac{\tau}{2} \big(c(\delta u_h^N,  u_h^{N})- c(\delta u_h^1,  u_h^{2})\big) \le &  C  \Co \|  \delta u_h^{N}-P_0( \delta u_h^{N})\|^2+ C \Co \|   u_h^{N}\|^2 \\
&+ C  \Co \|  \delta u_h^{1}-P_0( \delta u_h^{1})\|^2+  C \Co \| u_h^{2}\|^2.
\end{alignat*}
Hence,  we arrive at
\begin{equation*}
D_1  \le  C \Co  \sum_{n=1}^{N-1}  \|  \delta u_h^{n+1}-P_0( \delta u_h^{n+1})\|^2 +  C \Co \sum_{n=1}^{N-1}  \| \delta \delta u_h^{n+1}\|^2+ C \Co \| u_h^{N}\|^2+  C \Co \|   u_h^{N}\|^2.
\end{equation*}
Which combined with the estimate for $D_2$ gives
\begin{equation*}
S_1  \le  C \Co  \sum_{n=1}^{N-1}  \|  \delta u_h^{n+1}-P_0( \delta u_h^{n+1})\|^2 +  C \Co \sum_{n=1}^{N-1}  \| \delta \delta u_h^{n+1}\|^2+ C \Co \| u_h^{N}\|^2+  C \Co \|   u_h^{2}\|^2.
\end{equation*}

\flushleft
Using \eqref{614} and combining the above inequalities we get
\begin{alignat*}{1}
\frac{1}{2}\|u_h^{N}\|^2 - \frac{1}{2}\|u_h^{1}\|^2 + ||| \overline{u}_h|||^2 \le &  \frac{1}{16}  ||| \overline{u}_h|||^2 + 8 \tau \sum_{n=1}^{N-1} \|L^{n+1}\|_h^2+\frac{c_L}{32} \tau  \sum_{n=1}^{N-1} \|u_h^{n+1}\|^2+ C \Co \| u_h^{N}\|^2   \\
 & + C \Co(1+\gamma) \sum_{n=1}^{N-1}  \|  \delta \delta u_h^{n+1}\|^2+  C \Co  \sum_{n=1}^{N-1}  \|  \delta u_h^{n+1}-P_0( \delta u_h^{n+1})\|^2 \\
 &+ C \Co \|   u_h^{2}\|^2.
\end{alignat*}
Applying  \eqref{cor615-1} we get 
\begin{alignat*}{1}
\frac{1}{2}\|u_h^{N}\|^2 - \frac{1}{2}\|u_h^{1}\|^2 + ||| \overline{u}_h|||^2 \le &  (\frac{1}{16}+ C \Co^2)  ||| \overline{u}_h|||^2 + (8 +C \Co) \tau \sum_{n=1}^{N-1} \|L^{n+1}\|_h^2\\
&+(\frac{c_L}{32} + \tau C \Co \|\nabla \beta\|_{\infty}^2) \tau  \sum_{n=1}^{N-1} \|u_h^{n+1}\|^2+ C \Co \| u_h^{N}\|^2+  C \Co \|   u_h^{2}\|^2.
\end{alignat*}
\flushleft
Taking $\Co$ sufficiently small we arrive at
 \begin{alignat*}{1}
\frac{1}{4}\|u_h^{N}\|^2 + \frac{1}{2}  ||| \overline{u}_h|||^2 \le &   \frac{1}{2}\|u_h^{1}\|^2 +9 \tau \sum_{n=1}^{N-1} \|L^{n+1}\|_h^2\\
&+(\frac{c_L}{32} + \frac{\tau}{4} \|\nabla \beta\|_{\infty}^2) \tau  \sum_{n=1}^{N-1} \|u_h^{n+1}\|^2 +  \frac{1}{4} \|   u_h^{2}\|^2.  
\end{alignat*}
The inequality \eqref{thmCN2} follows from the above inequality and the discrete Gronwall's inequality \eqref{gronwall}.

{\bf Case 3: $\Pe  \ge 1$, $p \ge 1$}.
We use the same estimates for $S_2$ and $S_3$ as in Case 2 above. Then
inspecting the proof of the estimate for $S_1$ in Case 2 we see that we could instead have shown.,
\begin{equation*}
S_1  \le  C \sqrt{\tau}(\Co_{4/3})^{3/2}  \sum_{n=1}^{N-1}  \|  \delta u_h^{n+1}\|^2 +  \epsilon \sum_{n=1}^{N-1}  \| \delta \delta u_h^{n+1}\|^2+ C \Co \| u_h^{N}\|^2+  C \Co \|   u_h^{2}\|^2,
\end{equation*}
where $\epsilon$ is a sufficiently small number.  Now using  \eqref{cor615-2} and \eqref{lemma818-2} with $\max\{ \Co, \Co_{4/3}\}$ sufficiently small we get. 
\begin{alignat*}{1}
S_1 \le   C \Co \| u_h^{N}\|^2+  C \Co \|   u_h^{2}\|^2+   C \Co |||\overline{u}_h|||^2 + 9\tau   \sum_{n=1}^{N-1} \|L^{n+1}\|_h^2   +  \tau \sum_{n=1}^{N-1} \|u_h^{n+1}\|^2.
\end{alignat*}
Then we can proceed as we did in the proof of Case 2 to prove  \eqref{thmCN2}.
\section{Numerical examples}\label{sec:numerics}
We consider the methods applied to the pure transport
  problem. That is the methods obtained in the limit of vanishing diffusion:
the second order backward differentiation with
extrapolation (BDF2) and the second order Adams-Bashforth (AB2) scheme, both
of which are covered by the above analysis. For piecewise affine
approximation ($P_1$) we use the hyperbolic CFL, $\tau = Co\, h$
($\|\beta\|_\infty =1$) and for
piecewise quadratic approximation we use the 4/3-CFL, $\tau = Co
h^{4/3}$. Numerical experiments show that with $P_1$ approximation the
methods are stable under hyperbolic CFL, only when $\gamma>0$
(i.e. when stabilization is present). We also observed that for $P_2$
the 4/3-CFL is necessary for all $\gamma \ge 0$. The
values of Courant numbers and stabilization parameters used for the
different methods are given in
Table \ref{tb:param}. We stress that these values are not limit
values for stability for each case, but rather values that produced good results in
all the experiments performed.

We also consider a
numerical example using the third order Adams-Bashforth (AB3) scheme,
a scheme with non-trivial imaginary stability boundary \cite{GFR15} that is
expected to be stable under hyperbolic CFL, independent of the value
of $\gamma$. This also turns out to be
the case. For this method we show the results both for the stabilized
and the unstabilized method to show that even though the time
integrator is stable and boundary conditions are imposed weakly,
strong gradients destroy the solution globally irrespective of
polynomial approximation order unless stabilization is added. 

\begin{table}[htb]
\centering
\begin{tabular}{|c|c|c|}
\hline
Method & $Co$ &  $\gamma$ \\
\hline
BDF2/P1 & $0.15$ & $0.01$ \\
BDF2/P2 & $0.05$ & $0.005$ \\
AB2/P1 & $0.3$ & $0.01$ \\
AB2/P2 & $0.1$ & $0.005$ \\
AB3/P2 & $0.025$ & $0.001$ \\
AB3/P3 & $0.025$ & $0.0003$\\
\hline
\end{tabular}
\vspace{0.5cm}
\caption{Table showing the Courant numbers and stabilization
  parameters used for the different methods}\label{tb:param}

\end{table}
For the first example we consider transport in the disc $\Omega:= \{(x,y) \in \mathbb{R}^2 :
x^2+y^2  < 1\}$ under the velocity field $\beta = (y,-x)$. Approximations
are computed on a series of unstructured mehes with $nele=40,80,160,320$ elements along the disc perimeter. We let
$f=0$ and consider two different functions $u_0$ as  initial data.
One is smooth,
\begin{equation}\label{eq:u0_smooth}
u_0 = e^{-30((x-0.5)^2+y^2)}
\end{equation}
and one is rough 
\begin{equation}\label{eq:u0_rough}
\tilde u_0 = \left\{\begin{array}{l}
1 \quad \sqrt{(x+0.5)^2 +y^2} <0.2\\
0 \quad \mbox{otherwise}
\end{array} \right. .
\end{equation}
The velocity field simply turns the disc with the initial data and we
compute one turn so that the final solution should be equal to the
inital data. Two numerical experiments are considered, compute $u$ for
the initial data $u_0$ and $u_0 + \tilde u_0$.
\begin{figure}[t]
\centering
\hspace{-0.5cm}
\includegraphics[width=0.25\linewidth, angle=90]{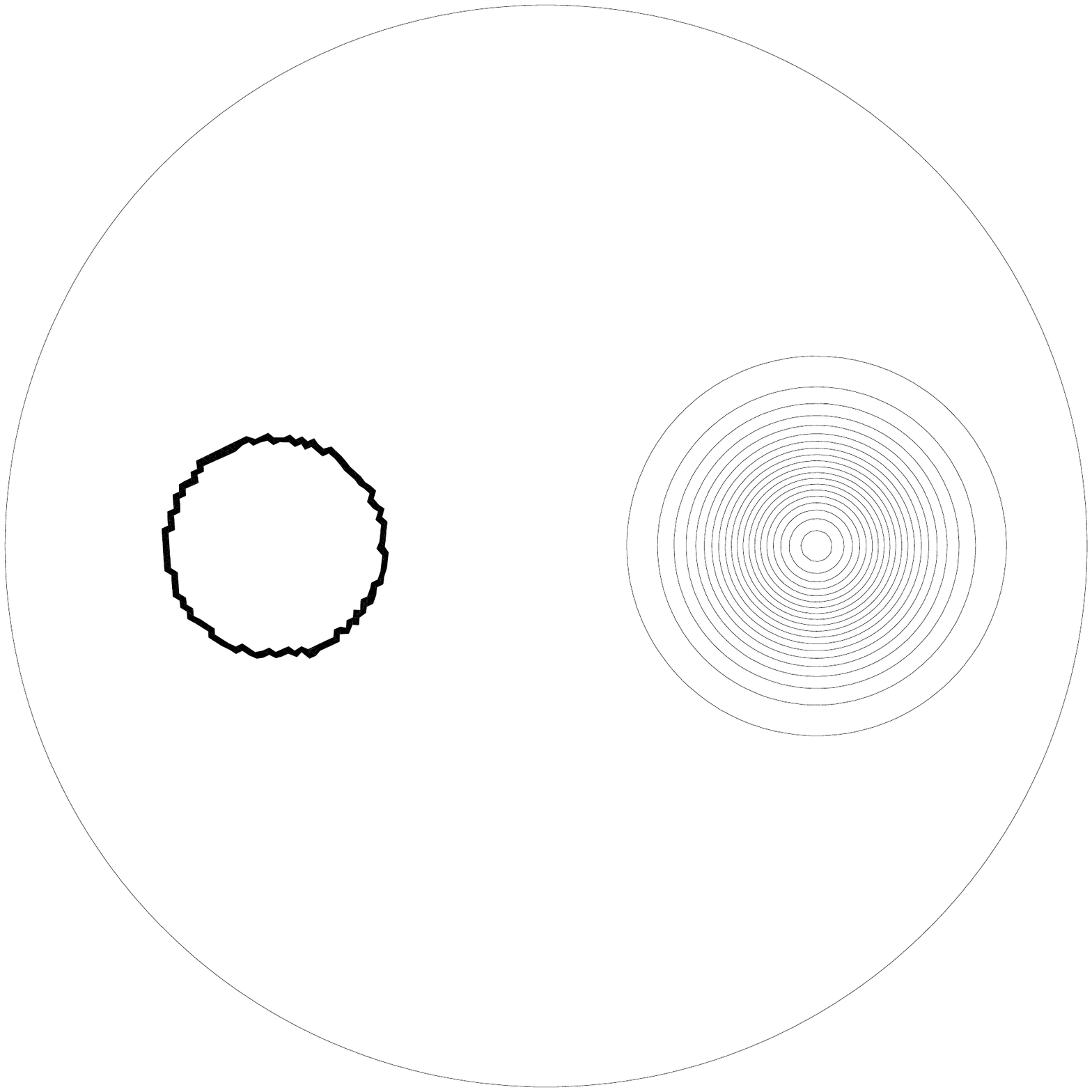} \hspace{-1.5cm}
\includegraphics[width=0.25\linewidth, angle=90]{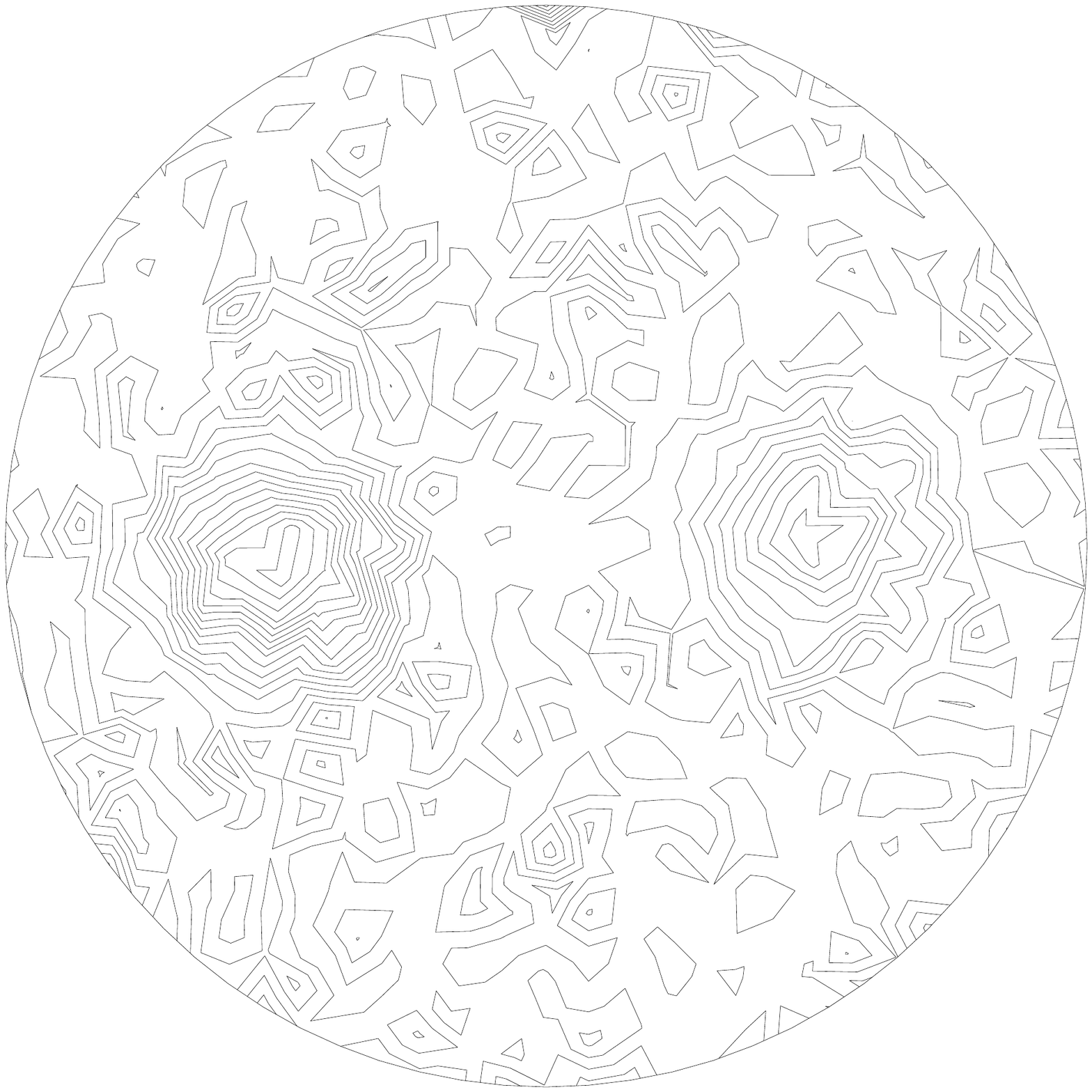}\hspace{-1.5cm}
\includegraphics[width=0.25\linewidth, angle=90]{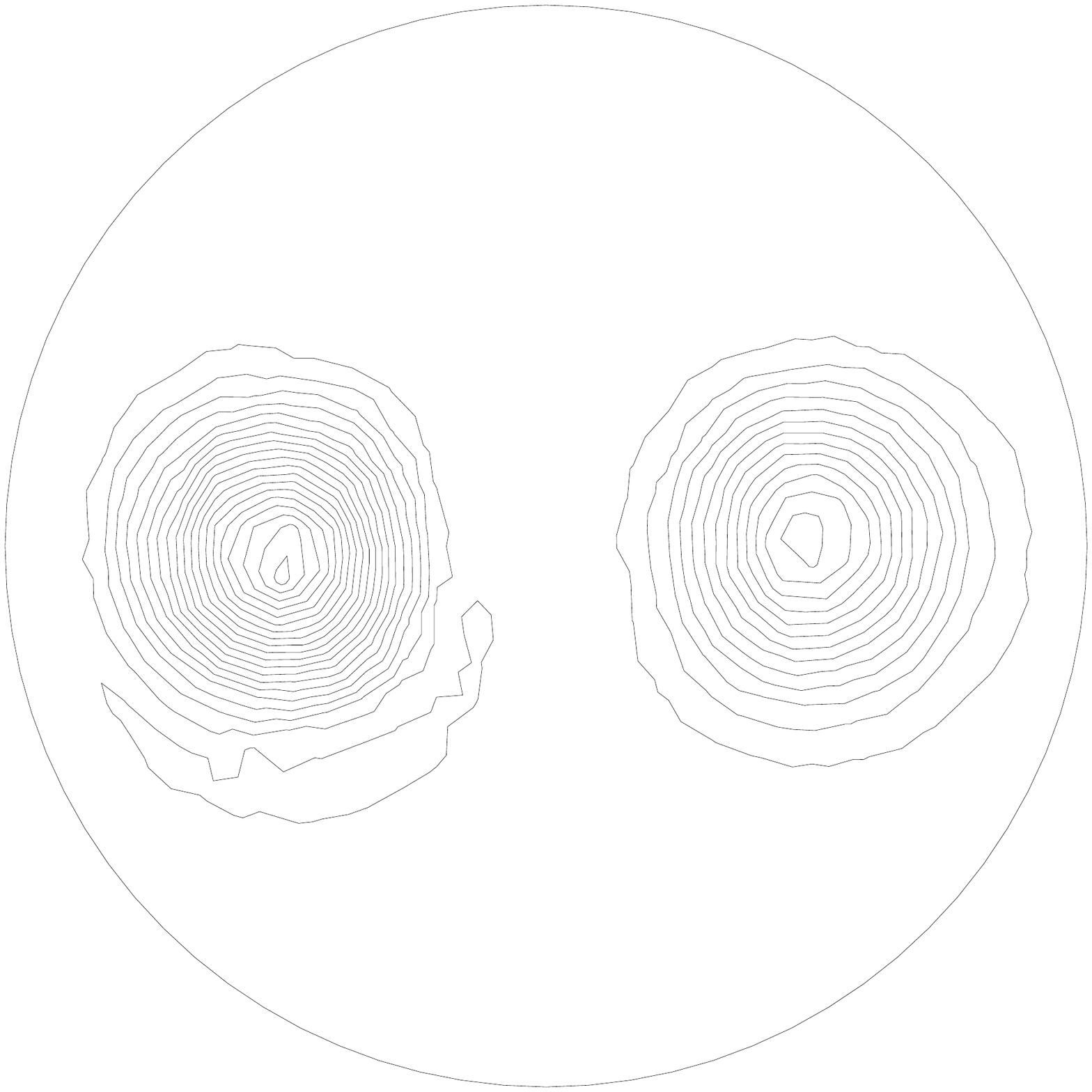}
\caption{From left to right: rough initial data on fine mesh ($u_0 + \tilde u_0$), unstabilized
  solution final time solution (computed using Crank-Nicolson), stabilized final time solution (both $nele=80$, one turn).}
\label{fig:data1}
\end{figure}
We compute the global error in the material derivative over the space
time domain, for BDF2
\[
\left(\tau\sum_{n=2}^N \|D_\tau u_h^{n+1} + \beta \cdot \nabla \tilde
  u_h^{n+1}\|^2 \right)^{\frac12},\quad \mbox{and for AB2, } \left( \tau\sum_{n=2}^N \|\tau^{-1} \delta u_h^{n+1} + \beta \cdot \nabla \hat u_h^{n+\frac12}\|^2 \right)^{\frac12}.
\]
In all graphics the material derivative is indicated by circle markers.
We also report the global
$L^2$-norm of the error at the final time, indicated by square markers. In the case where both the rough and the
smooth initial data are combined we compute the error obtained in the
smooth part, i.e. the $L^2$-norm over $\{(x,y) \in \Omega: x>0\}$. This
local error is indicated by triangle markers.

 In Figure
\ref{fig:data1} we show in the left panel the smooth and rough initial
data ($u_0 + \tilde u_0$). In the
middle panel the solution after one turn without stabilization
(computed using implicit Crank-Nicolson) and in
the right panel the solution after one turn with stabilization, in
both cases 
$nele =80$. We see
that the sharp layers are strongly smeared on this coarse mesh when
the stabilized method is used, but contrary to the unstabilized case
the smooth part of the solution is accurately captured. 

 In Figure \ref{fig:conv_smooth1} we
compare the convergence of the BDF2 and AB2 methods with
$P_1$ and $P_2$ elements for the smooth initial data. 
The convergence rates predicted by theory for both
stabilized methods and approximation spaces are verified
both for the $L^2$-error and in the material derivative. Both methods
have very similar errors,  In Figure \ref{fig:conv_rough2} we see that in the
presence of rough portions in the solution the stabilized methods
still have optimal convergence in the $L^2$-norm in the part where the
solution is smooth. We also observe that thanks to the stabilization
the material derivative has only moderate growth under refinement,
less than the $O(h^{-\frac12})$ predicted by theory. This is known not to be true for cG methods
without stabilization.
\begin{figure}[t]
\centering
\includegraphics[width=0.45\linewidth]{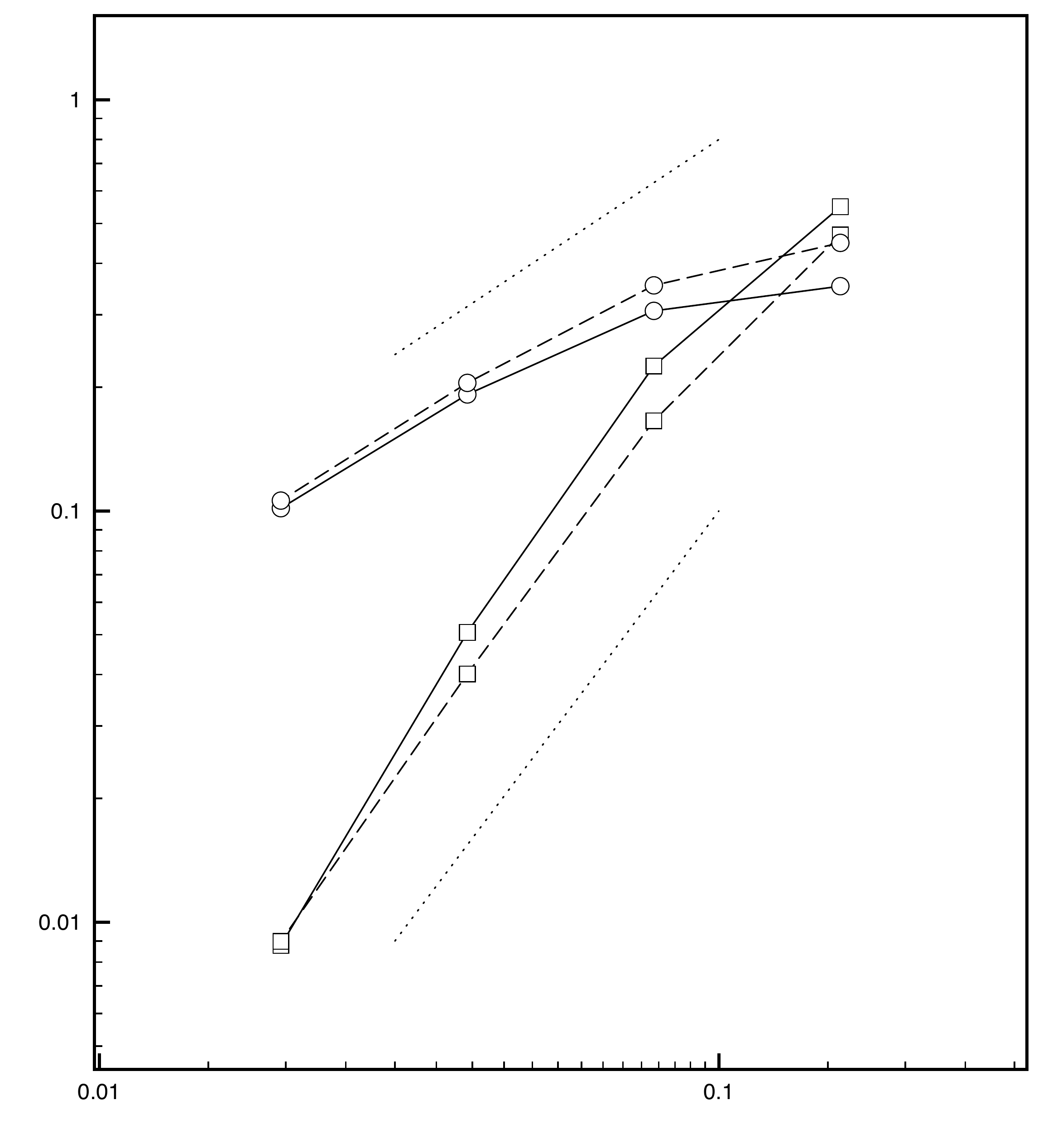}
\includegraphics[width=0.45\linewidth]{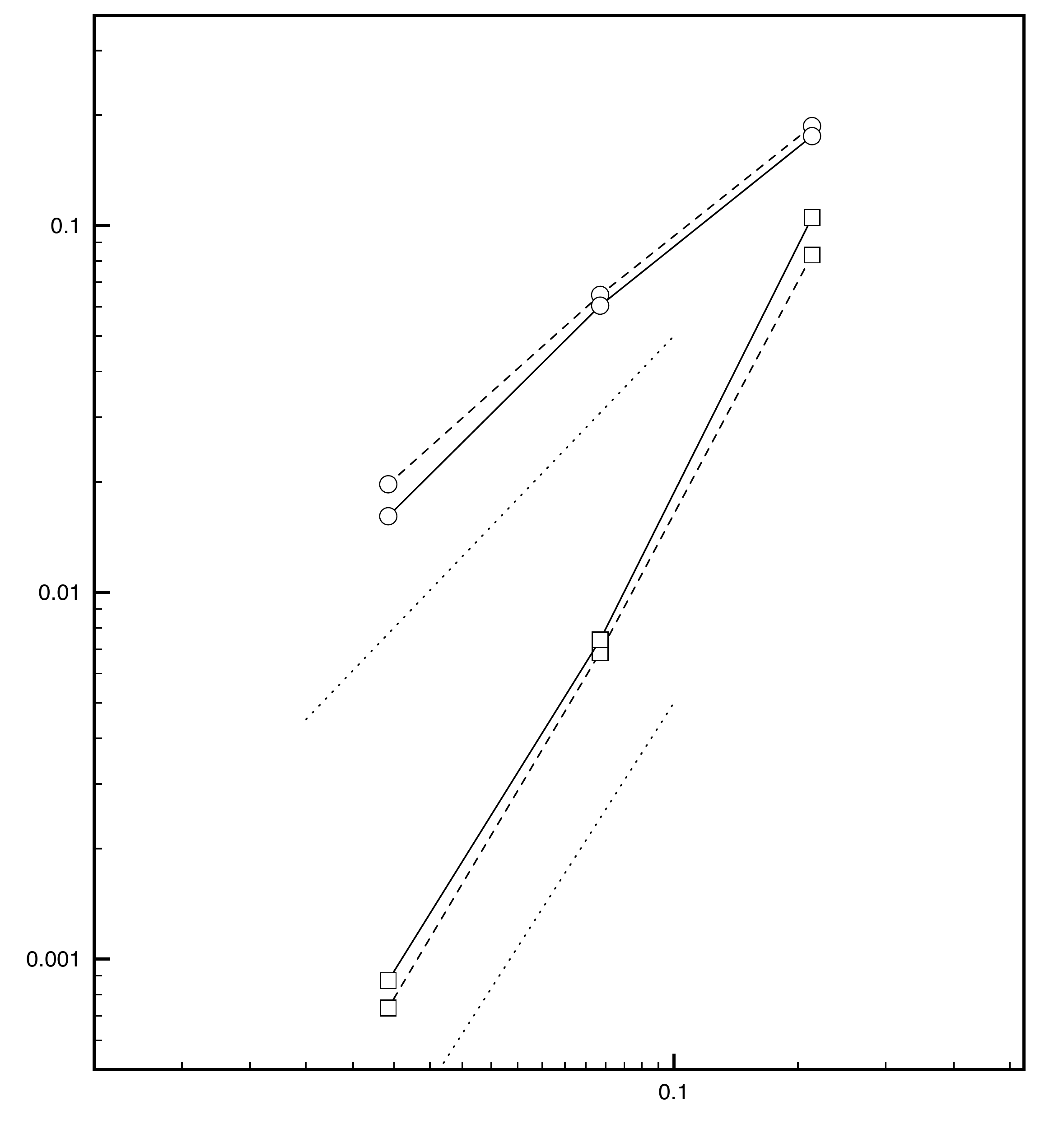}
\caption{Comparison BDF2 (full line) and AB2 (dashed
  line) method with $P_1$ (left) and $P_2$ (right) approximation,
  with globally smooth initial data (equation \eqref{eq:u0_smooth}). The
  error in
  material derivative has square markers. The
 global $L^2$-error has circle markers. The dotted reference lines
 have slope $1,2$ from top to bottom in the left plot and $2,3$ in the
 right plot.}
\label{fig:conv_smooth1}
\end{figure}
\begin{figure}[t]
\centering
\includegraphics[width=0.45\linewidth]{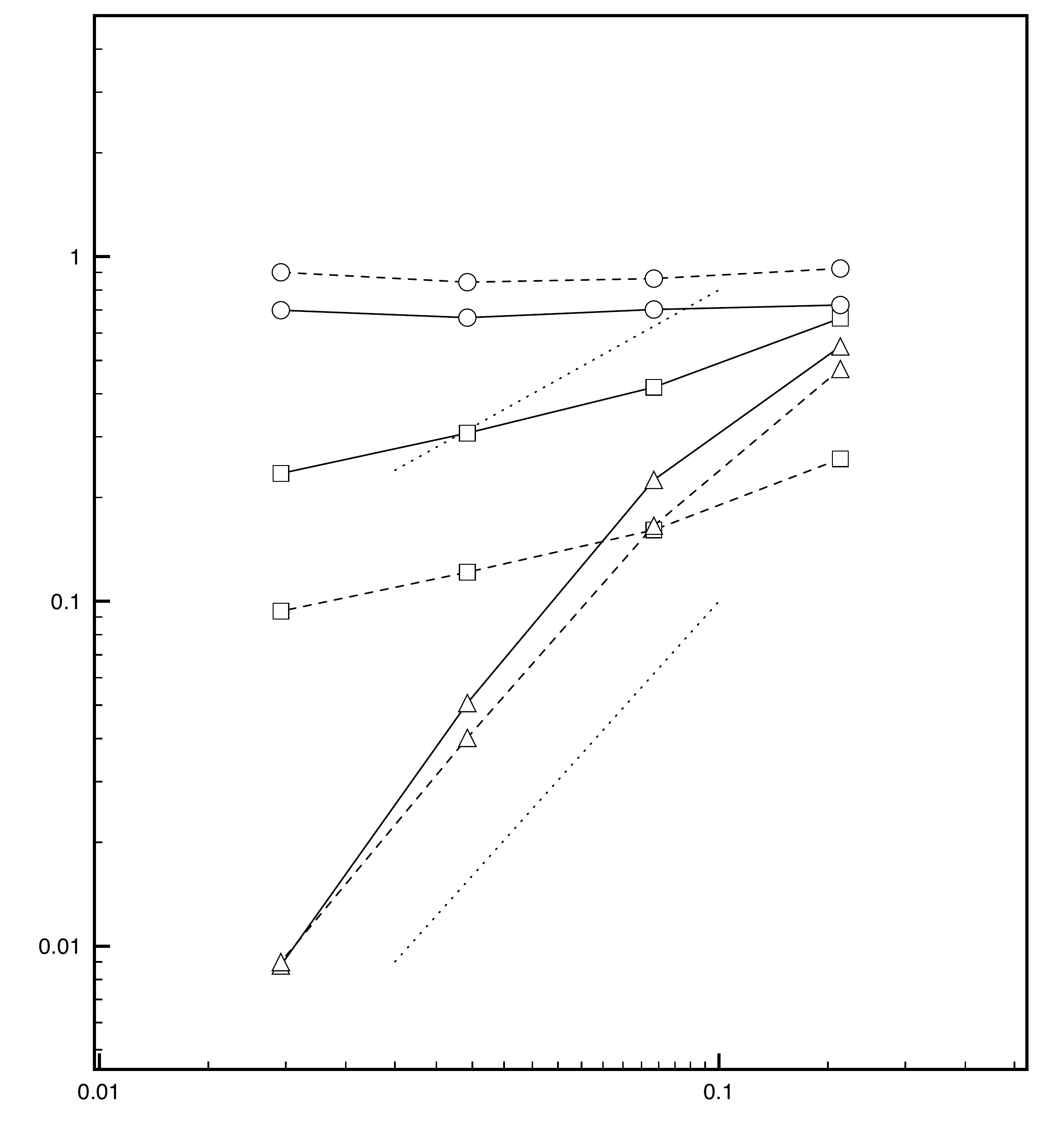}
\includegraphics[width=0.45\linewidth]{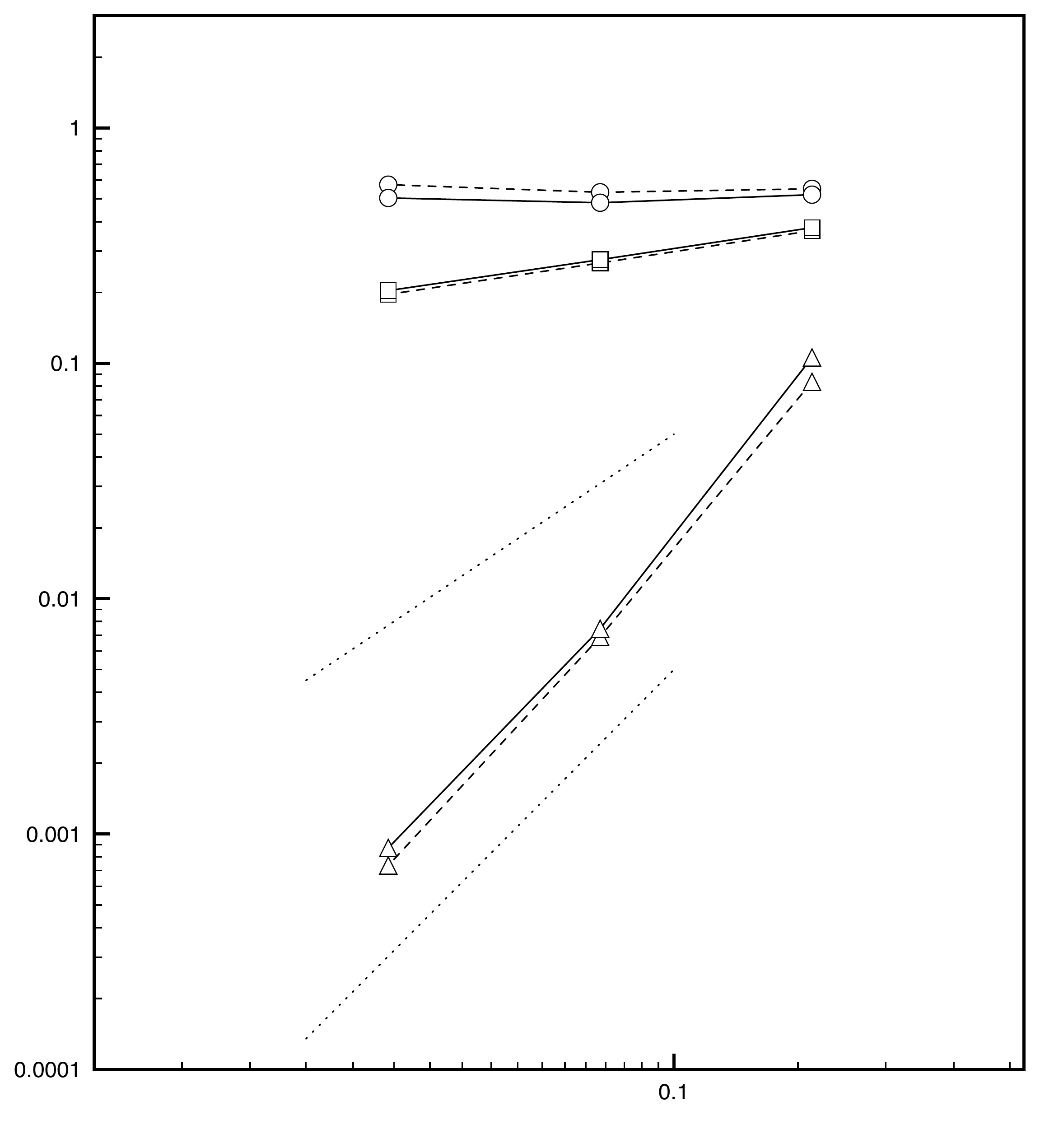}
\caption{Comparison BDF2 (full line) and AB2 (dashed
  line) method with $P_1$ (left) and $P_2$ (right). Initial data
  from figure \ref{fig:data1} (left plot). The
  error in
  material derivative has circle markers. The
 global $L^2$-error has square markers and the local $L^2$-error has triangle markers . The dotted reference lines
 have slope $1,2$  from top to bottom in the left plot and $2,3$ in
 the right.}
\label{fig:conv_rough2}
\end{figure}
\subsection{An example with inflow and outflow and weakly imposed
  boundary conditions}
Here we consider transport in the unit square with $\beta =
(1,0)^T)$. Structured meshes with $nele = 40,80,160,320$ elements on
each side are used. The initial data consists of a cylinder of radius $r=0.2$
centered in the middle of the square and a Gaussian centered on the
left boundary (See Figure \ref{fig:data2}, left plot). The exact
shapes are the same as those of the previous example, \eqref{eq:u0_smooth} and \eqref{eq:u0_rough}. We compute the solution over the interval $(0,1]$ so
that the cylinder leaves the domain at $t=0.7$ and at $t=1$ the
Gaussian is centered at on the right boundary (See Figure \ref{fig:data2}, right plot). Observe that from
$t=0.7$ the solution is smooth. The time dependent inflow boundary
condition is imposed weakly. The convergence of the $L^2$-error at final times for
the BDF2 and AB2  approaches is shown in Figure \ref{fig:conv_tube}
($h=1/nele$, $nele = 40,80,160,320$). We see
that for both methods the $P_1$ and $P_2$ approximations have
optimal convergence to the smooth final time solution, which is known
not to hold for the cG method without stabilization. This will be
verified in the next section.
\begin{figure}[t]
\centering
\hspace{-0.75cm}
\includegraphics[width=0.3\linewidth, angle=90]{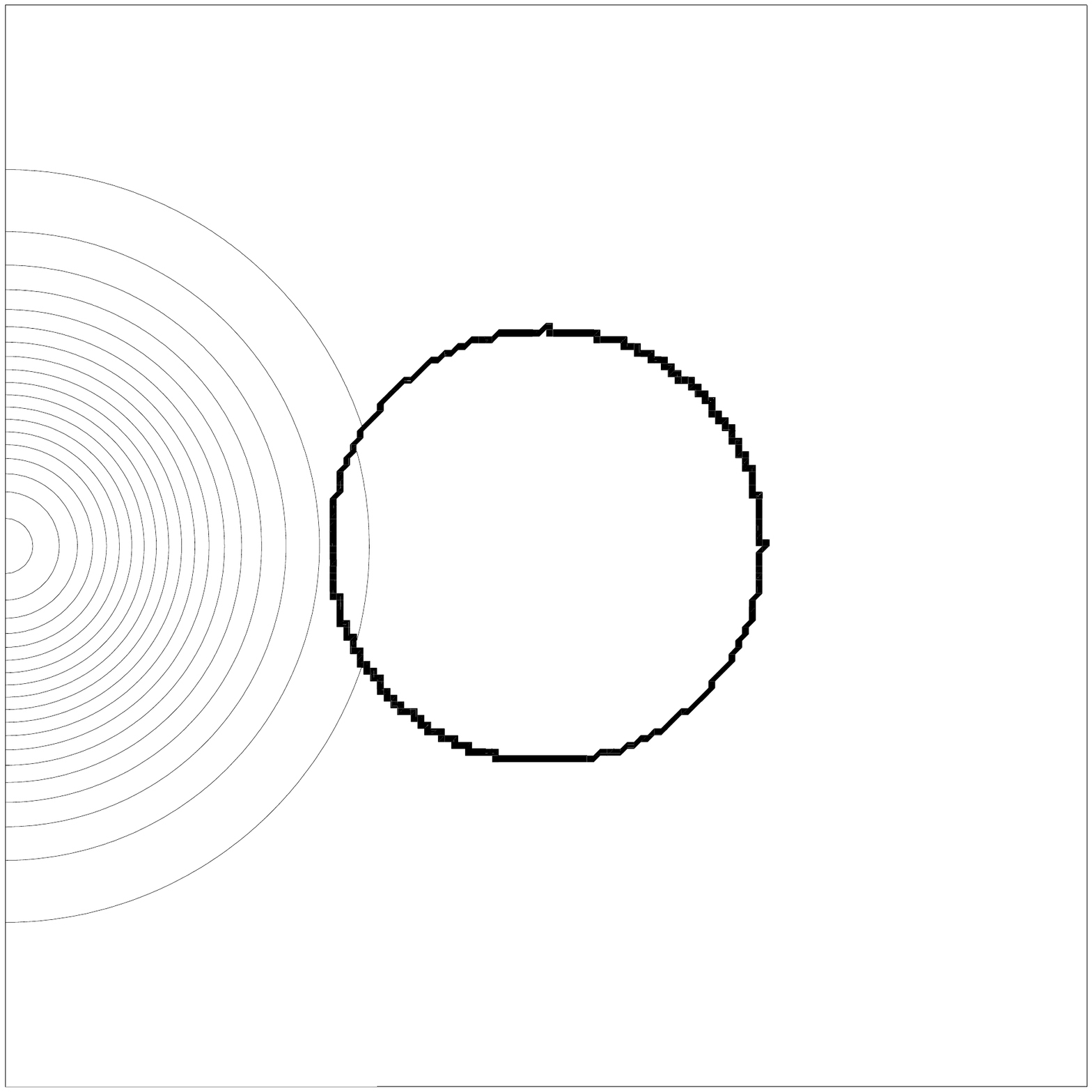} \hspace{-2cm}
\includegraphics[width=0.3\linewidth, angle=90]{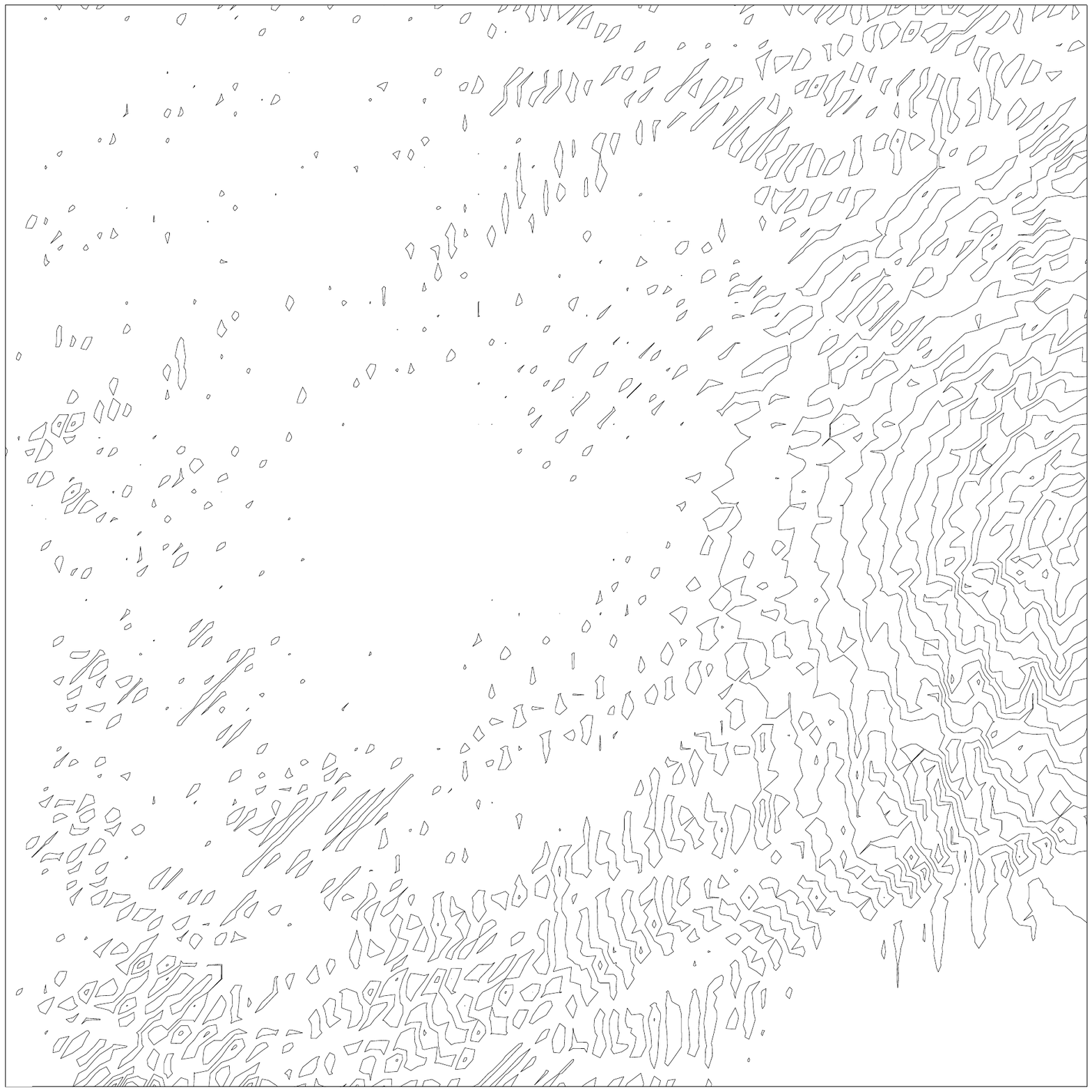}\hspace{-2cm}
\includegraphics[width=0.3\linewidth, angle=90]{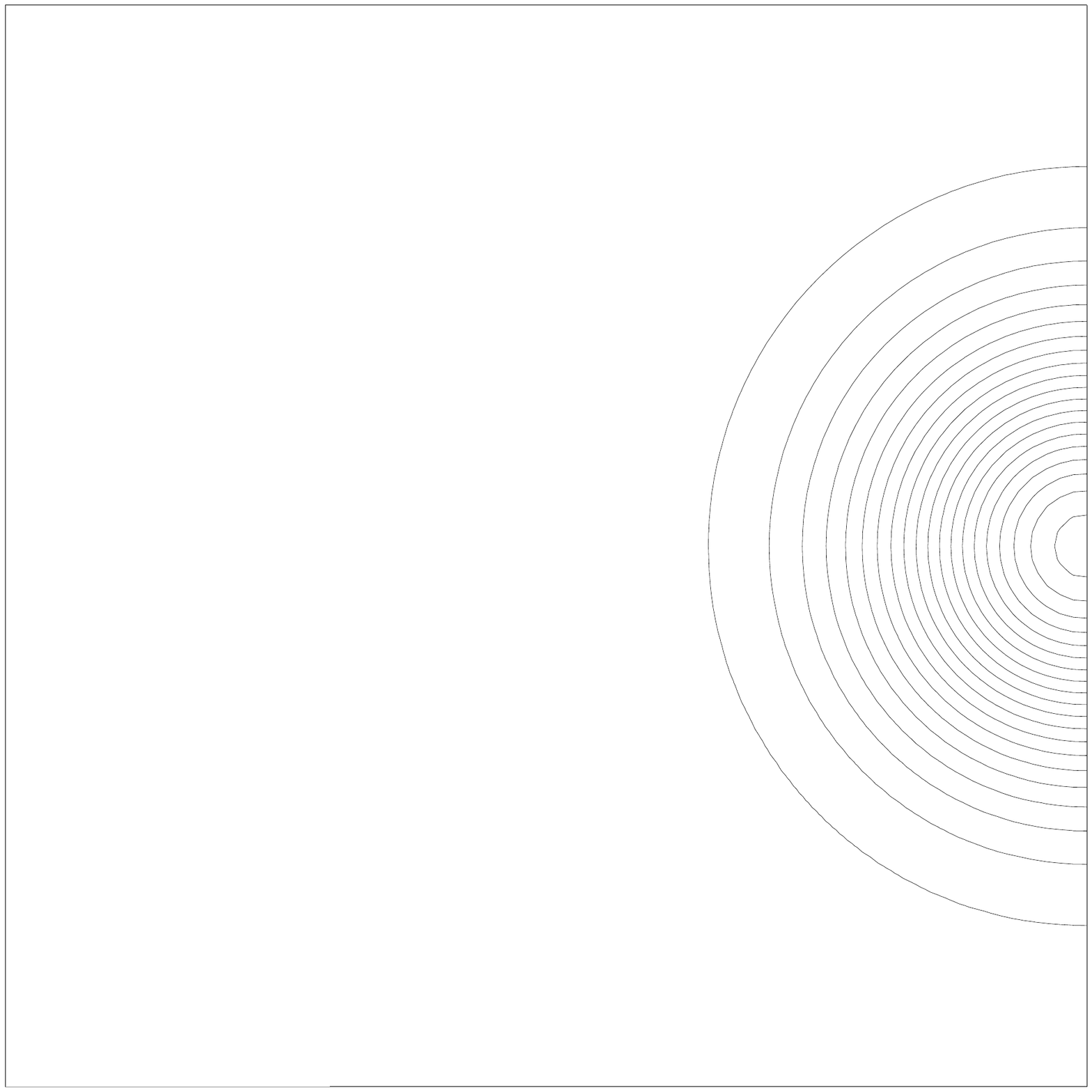}
\caption{From left to right: initial data on fine mesh,   unstabilized
  solution final time solution ($P_1$, computed using Crank-Nicolson), stabilized final time solution (both $nele=80$, final time $t=1$).}
\label{fig:data2}
\end{figure}
\begin{figure}[t]
\centering
\includegraphics[width=0.45\linewidth]{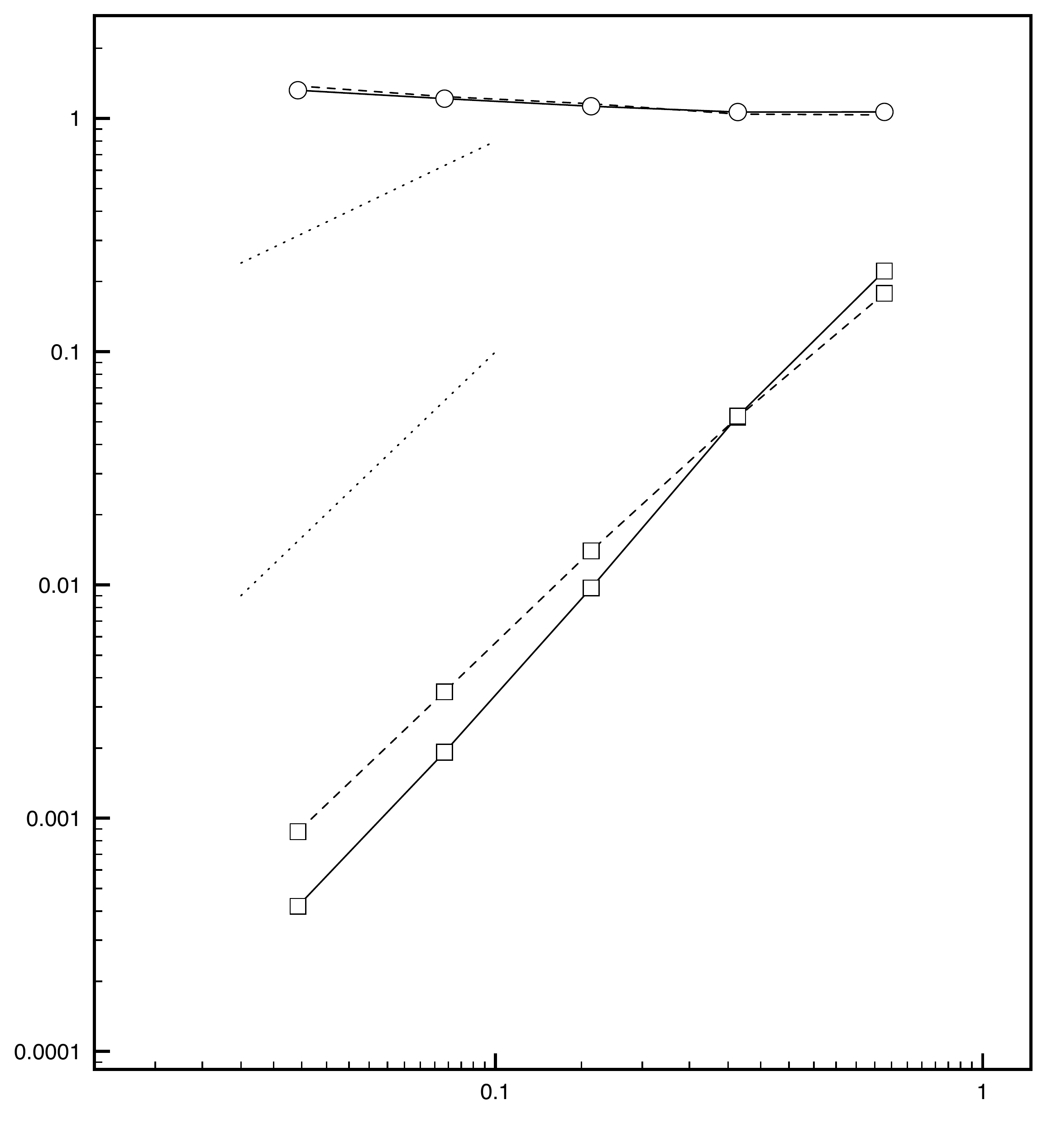}
\includegraphics[width=0.45\linewidth]{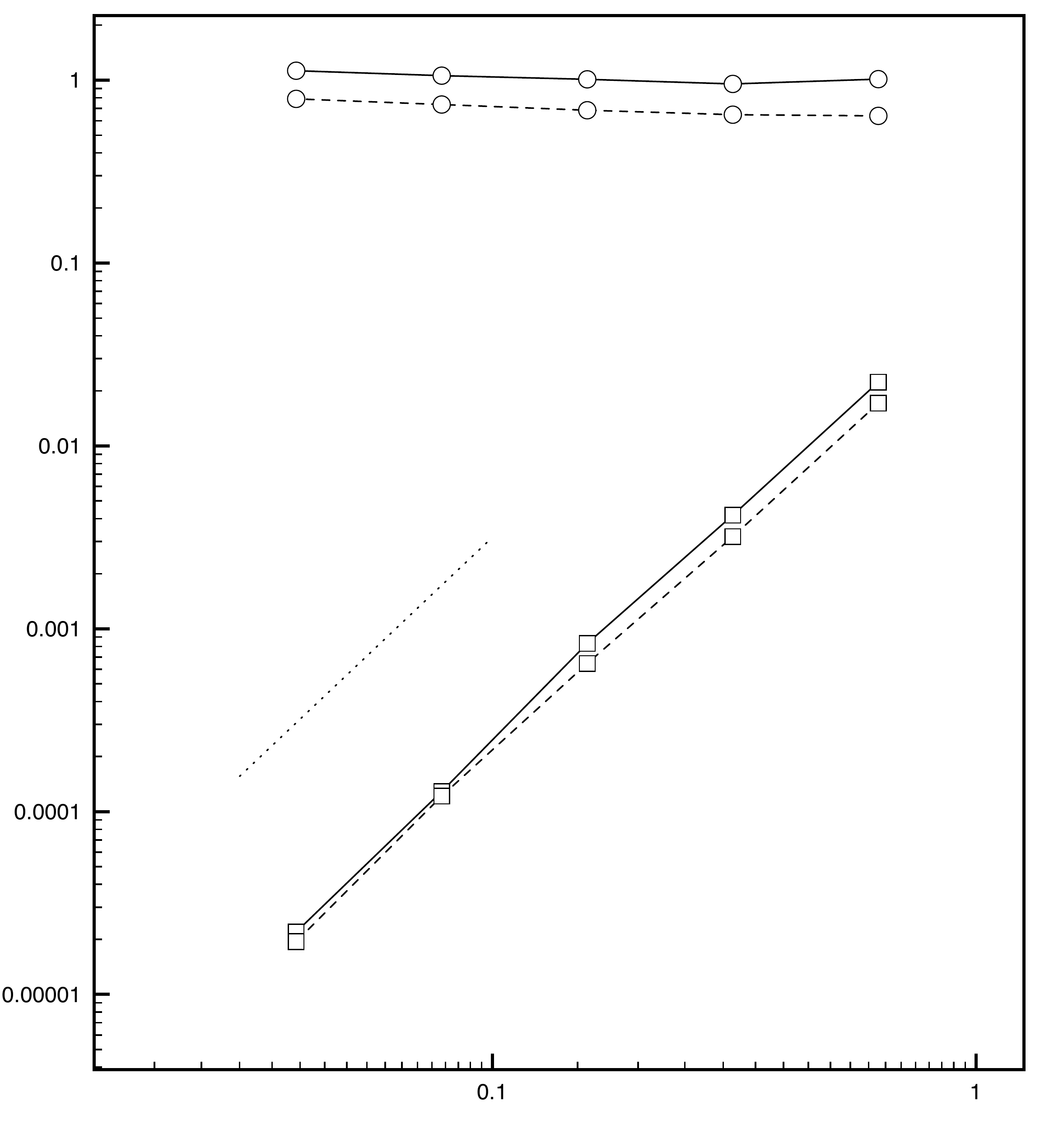}
\caption{Comparison BDF2 (full line) and AB2 (dashed
  line) method with $P_1$ (left) and $P_2$ (right).  Initial data
  from figure \ref{fig:data2} (left plot). The
  error in
  material derivative has circle markers. The
 global $L^2$-error has square markers and the local $L^2$-error has triangle markers. The dotted reference lines
 have slope $1,2$  from top to bottom in the left plot and $2.5$ in
 the right.}
\label{fig:conv_tube}
\end{figure}
\subsection{Higher order time integrator: Adams-Bashforth 3}
Here we consider the same test case as in the previous section, but
using the third order Adam-Bashforth scheme. In this case the scheme
is similar to AB2, but the extrapolation takes the form
\[
\hat v^{n+1} :=  \frac{23}{12} v^{n}-\frac{16}{12} v^{n-1}+ \frac{5}{12} v^{n-2}.
\]
For this test case we compare the results with or without
stabilization. We note that since the scheme has non-trivial imaginary
stability boundary, both the stabilized and unstabilized methods are
expected to be $L^2$-stable. This is also verified by the graphics in
Figure \ref{fig:conv_tube_AB3}. The Galerkin FEM without stabilization is distinguished
by filled markers in the graphics.
In Figure \ref{fig:conv_tube_AB3}, left plot, we present the result for $P_2$
finite elements. It is clear that the solution of the stabilized
method satifies
the $O(h^{2.5})$ bound predicted by theory (illustrated by the lower
  dotted line). Without stabilization the
method has approximately $O(h^{\frac12})$ (upper dotted line) convergence for the smooth
final time solution.

In the right plot we present the result for $P_3$ finite
elements. Also here the stabilized method has the expected
$O(h^{3.5})$ convergence (illustrated by the lower
  dotted line) and the unstabilized method fails to capitalize on the
  increased order of the method. Its order remains at $O(h^{\frac12})$
  (upper dotted line). As a consequence the stabilized method is more
  accurate by more than six orders of magnitude on the finest mesh.
\begin{figure}[t]
\centering
\includegraphics[width=0.45\linewidth]{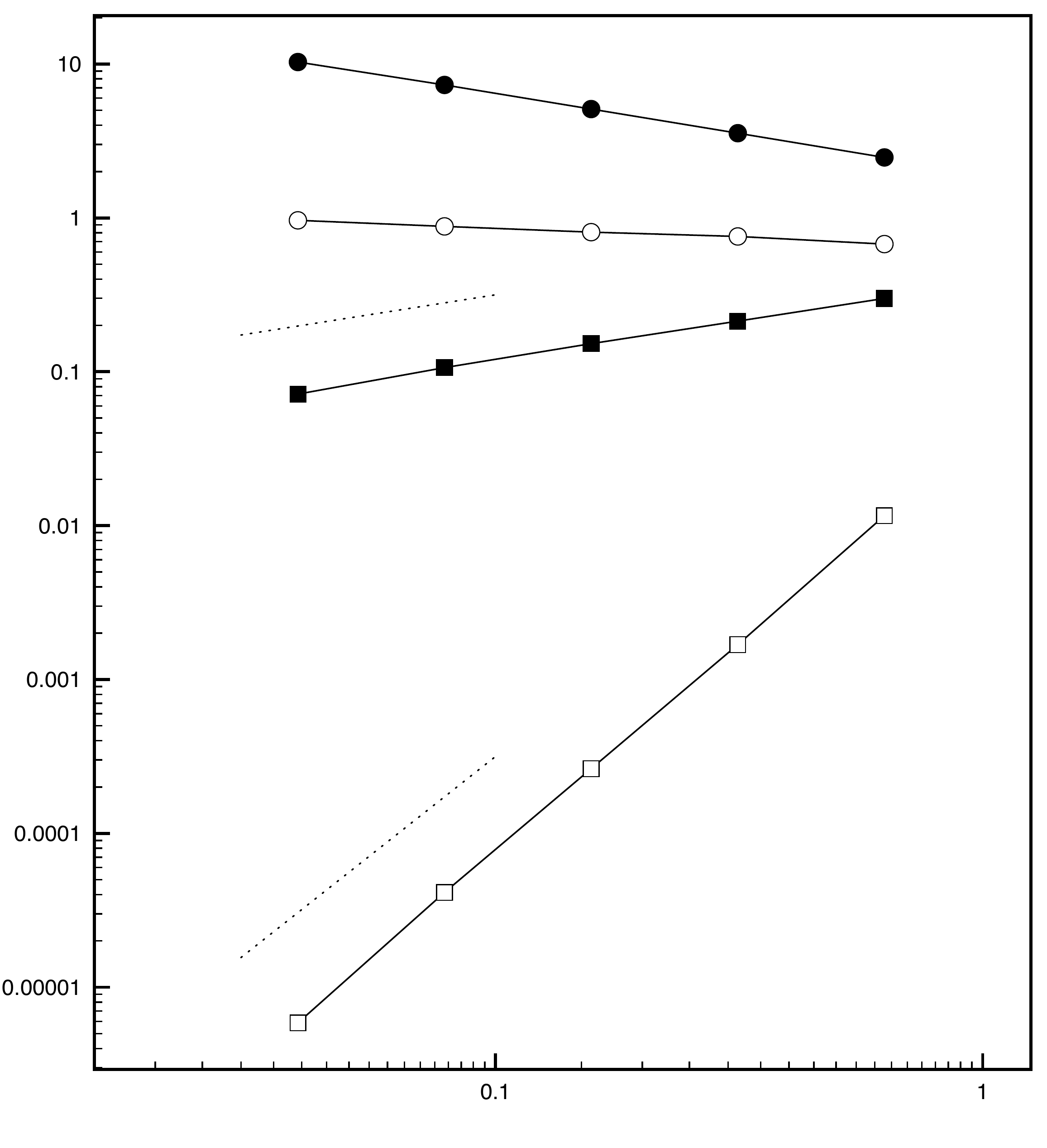}
\includegraphics[width=0.45\linewidth]{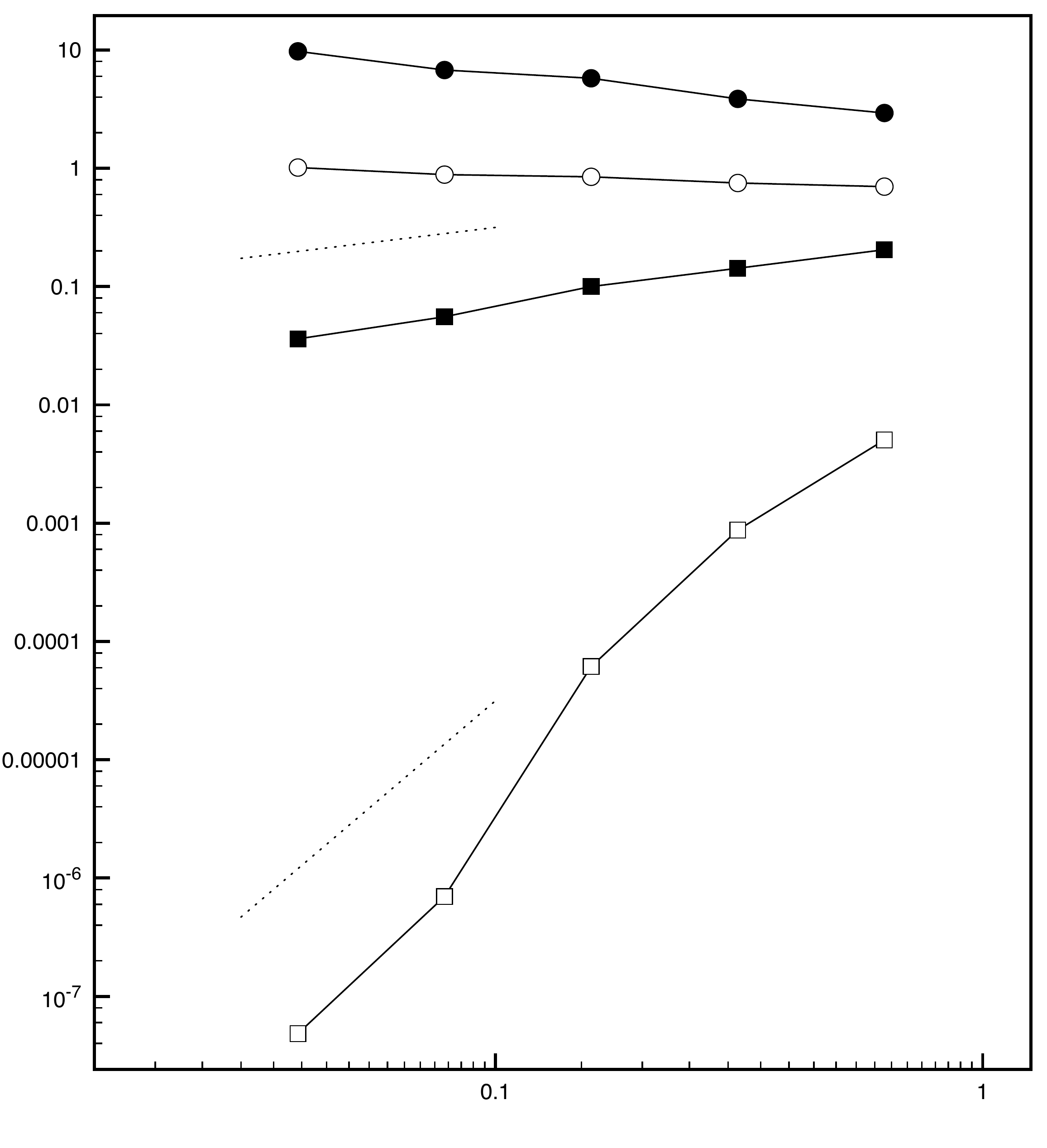}
\caption{Comparison AB3 method with $P_2$ (left plot) and $P_3$ (right
  plot) polynomial approximation.  Initial data
  from figure \ref{fig:data2} (left plot). The
  error in
  material derivative has circle markers. The
 global $L^2$-error has square markers. Filled markers indicate that
  no stabilization has been used. The upper dotted reference
 lines have slope $0.5$ in both graphics and the lower ones 
 have slopes $2.5$ (left) and $3.5$ (right).}
\label{fig:conv_tube_AB3}
\end{figure}
\section{Concluding remarks}
In this paper we have considered the use of implicit-explicit time
integrators together with symmetric stabilization methods. An
important advantage is that the often nonlinear convection term is
handled explicitly as well as the stabilization, which otherwise is
known to extend the system matrix. Two second order methods were
considered that are appealing in applications for their simplicity,
but that have trivial imaginary stability boundary in the limit of
vanishing diffusion. We prove that nevertheless these methods can be
used together with stabilized FEM (or upwind discontinuous Galerkin
method) under CFL conditions that allow for an optimal matching of
errors in space and time. The present work opens for several
interesting research venues such as the use of predictor-corrector
methods \cite{GFR15} in combination with stabilized FEM for first
order pde, or higher
order IMEX-schemes based on Adams-Bashforth/Adams-Moulton combinations
for singularly perturbed second order systems such as
convection--diffusion or the Navier-Stokes' equations.
\bibliographystyle{plain}
\bibliography{references-2}
\end{document}